\documentclass[final]{siamltex}

\usepackage{amsmath}
\usepackage{amssymb}
\usepackage{enumitem}
\usepackage{yhmath}
\usepackage{lipsum}
\usepackage{graphicx}
\usepackage{longtable}
\usepackage{booktabs}
\usepackage[
  pdftitle={An optimal polynomial approximation of Brownian motion},
  pdfauthor={James Foster},
]{hyperref}
\hypersetup{
    colorlinks=false,
}

\newcommand\blfootnote[1]{%
  \begingroup
  \renewcommand\thefootnote{}\footnote{#1}%
  \addtocounter{footnote}{-1}%
  \endgroup
}

\newlist{myitemize}{itemize}{3}
\setlist[myitemize]{label=\textbullet,leftmargin=0.4in}

\newcommand*\samethanks[1][\value{footnote}]{\footnotemark[#1]}

\newcommand{\cov}{\operatorname{cov}}
\newcommand{\var}{\operatorname{Var}}

\title{An optimal polynomial approximation of\\ Brownian motion}

\author{James Foster\blfootnote{M\MakeLowercase{athematical}\hspace{0.36mm} I\MakeLowercase{nstitute},\hspace{0.36mm} U\MakeLowercase{niversity\hspace{0.36mm} of}\hspace{0.36mm} O\MakeLowercase{xford},\hspace{0.36mm} W\MakeLowercase{oodstock}\hspace{0.36mm} R\MakeLowercase{oad},\hspace{0.36mm} O\MakeLowercase{xford},\hspace{0.36mm} OX2\hspace{0.36mm} 6GG,\hspace{0.36mm} UK.\\ \MakeLowercase{\hspace*{1.8em}james.foster@maths.ox.ac.uk,\hspace{0.625mm} terry.lyons@maths.ox.ac.uk,\hspace{0.625mm} harald.oberhauser@maths.ox.ac.uk.}}\thanks{Research supported by the Engineering and Physical Sciences Research Council [EP/N509711/1].}
        \and Terry Lyons\thanks{Supported\hspace{0.25mm} by\hspace{0.25mm} the\hspace{0.25mm} EPSRC\hspace{0.25mm} grant\hspace{0.25mm} DATASIG,\hspace{0.25mm} Alan\hspace{0.25mm} Turing\hspace{0.25mm} Institute\hspace{0.25mm} and\hspace{0.25mm} Oxford-Man\hspace{0.25mm} Institute.}
        \and Harald Oberhauser\samethanks}

\begin{document}
\maketitle

\begin{abstract}
In this paper, we will present a strong (or pathwise) approximation of standard
Brownian motion by a class of orthogonal polynomials. The coefficients that are obtained from the
expansion of Brownian motion in this polynomial basis are independent Gaussian random variables.
Therefore it is practical (requires $N$ independent Gaussian coefficients) to generate an approximate
sample path of Brownian motion that respects integration of polynomials with degree less than $N$.
Moreover, since these orthogonal polynomials appear naturally as eigenfunctions of the Brownian
bridge covariance function, the proposed approximation is optimal in a certain weighted $L^{2}(\mathbb{P})$ sense. \\
In addition, discretizing Brownian paths as piecewise parabolas gives a locally higher order numerical
method for stochastic differential equations (SDEs) when compared to the piecewise linear approach. 
We shall demonstrate these ideas by simulating Inhomogeneous Geometric Brownian Motion (IGBM).
This numerical example will also illustrate the deficiencies of the piecewise parabola approximation
when compared to a new version of the asymptotically efficient log-ODE (or Castell-Gaines) method.
\end{abstract}



\pagestyle{myheadings}
\thispagestyle{plain}
\markboth{J. FOSTER, T. LYONS AND H. OBERHAUSER}{POLYNOMIAL APPROXIMATION OF BROWNIAN MOTION}

\section{Introduction}

Brownian motion is a central object for modelling real-world
systems that evolve under the influence of random perturbations \cite{Brownianuseful}. In applications
where methods discretize Brownian motion, usually only increments of the path are
generated \cite{Higman}. In this setting, the best $L^{2}(\mathbb{P})$ approximation of Brownian motion that
is measurable with respect to these increments is given by the piecewise linear path
that agrees on discretization points \cite{CameronClark}. This motivates the following natural question:\vspace{1.5mm}
\textit{Are there better discrete approximations of Brownian motion than piecewise linear?}\vspace{1.5mm}
The next simplest approximant would be a piecewise polynomial, though it is not clear
whether this would be advantageous for tackling problems such as SDE simulation.
This paper can be viewed as a logical continuation of \cite{Multiscalebm}, where a polynomial wavelet
representation of Brownian motion was proposed. These wavelets were constructed to
capture certain ``geometrical features'' of the path, namely the integrals of the
Brownian motion against monomials. We shall investigate the practical applications
of these polynomials and their geometrical features in the numerical analysis of SDEs.\\ \vspace{-4mm}
\begin{figure}[h]\label{polynomialdiagams}
\centering
\includegraphics[width=0.975\textwidth]{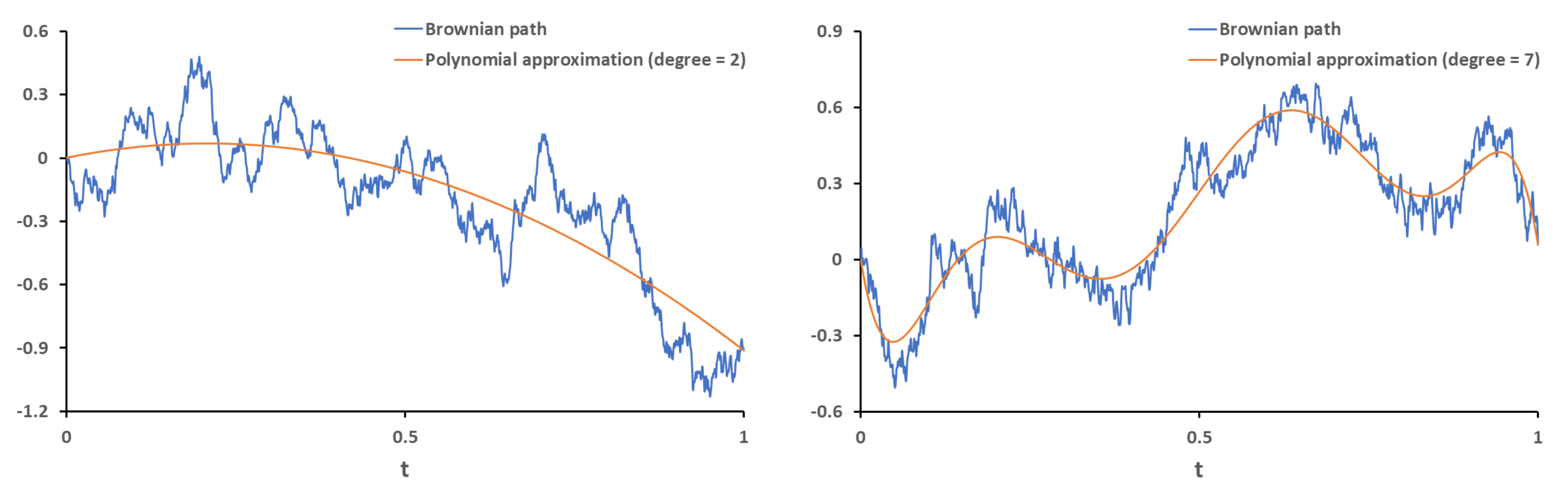}\vspace{-2mm}
\caption{Sample paths of Brownian motion with corresponding polynomial approximations.}\vspace{-3.mm}
\end{figure}\newpage\noindent
The paper is organised as follows. In Section 2, we shall state and prove the main
result of the paper (Theorem \ref{waveletthm}). This will be a Karhunen-Lo\`{e}ve theorem for the Brownian bridge,
where the orthogonal functions used in the approximation are polynomials.
Furthermore, we shall explicitly show that each basis function is proportional to a shifted $(\alpha, \beta)$-Jacobi polynomial but with the nonstandard exponents $\alpha = \beta = -1$.
This enables us to construct these orthogonal polynomials using recurrence relations,
or as the difference of two shifted Legendre polynomials whose degrees differ by two.
The resulting polynomial expansion of Brownian motion was independently discovered
by Habermann in \cite{Habermann}, where a sharp $L^{2}\hspace{-0.075mm}(\mathbb{P})$ convergence rate of $O\big(\frac{1}{\sqrt{n}}\big)$ is established\footnote{\hspace*{-0.25mm}A Matlab demonstration can be found at \href{https://www.chebfun.org/examples/stats/RandomPolynomials.html}{chebfun.org/examples/stats/RandomPolynomials.html}}.\vspace{0.5mm}
\noindent
In\hspace{0.35mm} Section\hspace{0.35mm} 3,\hspace{0.35mm} we\hspace{0.35mm} shall\hspace{0.35mm} investigate\hspace{0.35mm} some\hspace{0.35mm} significant\hspace{0.35mm} consequences\hspace{0.35mm} of\hspace{0.35mm} the\hspace{0.35mm} main\hspace{0.35mm} theorem.\vspace{0.75mm}
\begin{theorem}
Let $W$ denote a standard real-valued Brownian motion on $[\hspace{0.1mm}0,1]$.
Let $W^{n}$ be the unique $n$-th degree random polynomial with a root at $0$ and satisfying
\begin{align}
\int_{0}^{1}u^{k}\,dW_{u}^{n} & = \int_{0}^{1}u^{k}\,dW_{u}\hspace{0.125mm},\hspace{3mm}\text{for}\hspace{2mm}k = 0\hspace{0.25mm}, 1\hspace{0.25mm}, \cdots\hspace{0.25mm}, n - 1\hspace{0.25mm}.
\end{align}
Then\hspace{0.25mm} $W =\hspace{0.25mm} W^{n} + Z^{n}$\hspace{0.25mm}, where\hspace{0.25mm} $Z^{n}$ is a centered Gaussian process independent of\hspace{0.5mm} $W^{n}$.
\end{theorem}\smallbreak
\noindent The above theorem has a simple yet striking conclusion, namely that polynomials can
be unbiased approximants of Brownian motion. In addition, the first non-trivial case
($n=2$) already has interesting applications within the numerical analysis of SDEs.
One reason is that parabolas can capture the ``\hspace{0.25mm}space-time area'' of Brownian motion.\vspace{-7.5mm}
\begin{figure}[h]\label{paraboladiagram}
\centering
\includegraphics[width=0.975\textwidth]{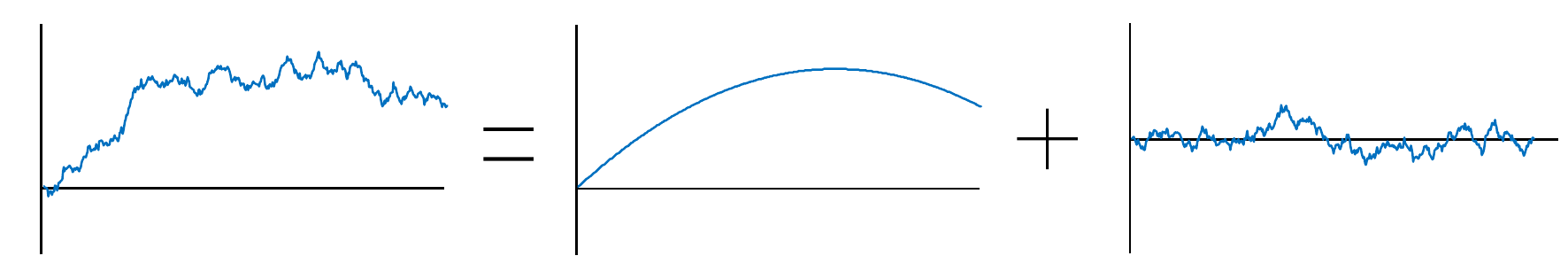}\vspace{-5mm}
\caption{Brownian motion can be expressed as a (random) parabola plus independent noise.
Moreover, the approximating parabola has the same increment and time integral as the original path.
}
\end{figure}\vspace{-4mm}\smallbreak\noindent
Therefore discretizing Brownian motion using a piecewise parabola gives a locally high
order methodology for numerically solving one-dimensional SDEs. However, since
certain triple iterated integrals of Brownian motion and time are partially matched by
these parabolas, we expect this method to have only an $O(h)$ rate of convergence
(where $h$ denotes the step size used). This gives motivation for the following theorem:\smallbreak
\begin{theorem}\label{introintegral}
Let\hspace{0.25mm} $\wideparen{W}$ be the (unique) quadratic polynomial with a root at $0$\hspace{0.25mm} and\vspace{-0.75mm}
\begin{align*}
\wideparen{W}_{1} = W_{1}\hspace{0.25mm}, \hspace{5mm}\int_{0}^{1}\wideparen{W}_{u}\, du = \int_{0}^{1}W_{u}\,du\hspace{0.25mm}.
\end{align*}
Then the following third order iterated integral of Brownian motion can be estimated:\vspace{-0.75mm}
\begin{align}\label{introintegralestimate}
\mathbb{E}\left[\hspace{0.25mm}\int_{0}^{1}W_{u}^{2}\,du\,\Big|\,W_{1}\hspace{0.25mm}, \int_{0}^{1}W_{u}\,du\hspace{0.25mm}\right]
= \int_{0}^{1}\wideparen{W}_{u}^{2}\,du + \frac{1}{15}\hspace{0.25mm}.
\end{align}
\end{theorem}\vspace{-3mm}\smallbreak
\noindent The above theorem can be directly incorporated into the stochastic Taylor
method as well as the log-ODE or Castell-Gaines method (see \cite{CastellGaines}, \cite{StocLie}). We will show that by
estimating this non-trivial iterated integral with its conditional expectation, we can
design numerical methods that enjoy high orders of both strong and weak convergence.
Specifically, for a general SDE that is driven by a one-dimensional Brownian motion
and governed by sufficiently regular vector fields (smooth with bounded derivatives),
the numerical methods that correctly utilize the above conditional expectation will
have a strong convergence rate of $O(h^{\frac{3}{2}})$ as well as a weak convergence rate of $O(h^{2})$.
\newpage\noindent
These high orders of convergence can also be achieved in the multidimensional setting
provided the vector fields governing the SDE satisfy certain commutativity conditions.
For example, this estimator has applications for simulating SDEs with additive noise:
\begin{align*}
dy_{t} & = f(y_{t})\,dt + \sigma\hspace{0.25mm}dW_{t}\hspace{0.25mm},
\end{align*}
where $f$ is a smooth vector field on $\mathbb{R}^{d}$, $\sigma > 0$ is constant and $W$ is now $d$-dimensional.
By considering Theorem \ref{introintegral}, we expect that $y_{t}$ is well approximated (for small $t$) by
\begin{align*}
\wideparen{y}_{t} + \frac{1}{30}\hspace{0.25mm}t^{2}\sigma^{2}\Delta f(y_0)\hspace{0.25mm},
\end{align*}
where $\wideparen{y}$ denotes the solution of the below ODE driven by a ``\hspace{0.25mm}Brownian parabola'' $\wideparen{W}$,
\begin{align*}
d\hspace{0.25mm}\wideparen{y}_{t} & = f(\hspace{0.25mm}\wideparen{y}_{t})\,dt + \sigma\hspace{0.25mm}d\wideparen{W}_{t}\hspace{0.25mm},\\[3pt]
\wideparen{y}_{0} & = y_{0}\hspace{0.25mm}.
\end{align*}
This parabola-driven ODE can then be discretized using a three-stage Runge-Kutta
method and the resulting SDE approximation shall be investigated in a future work.\medbreak
\noindent
Since these methods are based on the conditional expectation given by Theorem \ref{introintegral},
they are designed to minimize the leading error term within local Taylor expansions.
This sense of optimality is conceptually similar to that of the asymptotically efficient
SDE approximations developed by Clark \cite{Clark}, Newton \cite{Newton1, Newton2} and Castell \& Gaines \cite{CastellGaines}.
The key difference is that we are employing additional integral information about $W$.
Hence, this line of research could provide a further insight into the approximation of It\^{o} integrals using linear path information, where there already are a number of results concerning the computational complexity of methods (see \cite{Xiao}, \cite{LinearApprox} and \cite{Dickinson}).\\
Most notably, Tang and Xiao \cite{Xiao} consider the same triple iterated integral as in (\ref{introintegralestimate}) and present an asymptotically optimal approximation that performs well when a limited number of random variables are used (see Table 2 for these numerical results).
Whilst there are other senses of optimality (such as those discussed in \cite{Lee} and \cite{Spline})
that could be used when analysing the proposed approximations of Brownian motion and SDE solutions, we shall estimate errors in an $L^{2}(\mathbb{P})$ sense throughout the paper.
In particular, we can apply the main result to quantify the error of the new estimator.\medbreak
\begin{theorem}\label{introintegralvar}
Using the same notation as before, we have the following variance:
\begin{align*}
\var\left(\int_{0}^{1}W_{u}^{2}\,du\,\Big|\,W_{1}\hspace{0.25mm}, \int_{0}^{1}W_{u}\,du\right)
= \frac{11}{6300} + \frac{1}{180}\hspace{0.25mm}W_{1}^{2} + \frac{1}{175}\left(\hspace{0.25mm}\int_{0}^{1}W_{u}\,du - \frac{1}{2}\hspace{0.25mm}W_{1}\right)^{2}.
\end{align*}
\end{theorem}\bigbreak
\noindent
In Section 4, we demonstrate the applicability of these ideas to SDE simulation through various discretizations of Inhomogeneous Geometric Brownian Motion (IGBM)
\begin{align*}
dy_{t} = a(b-y_{t})\,dt+\sigma\hspace{0.125mm}y_{t}\,dW_{t}\hspace{0.25mm},
\end{align*}
where $a \geq 0$ and $b\in \mathbb{R}$ are the mean reversion parameters and $\sigma \geq 0$ is the volatility.\medbreak\noindent
In mathematical finance, IGBM is an example of a short rate model that can be both
mean-reverting and non-negative. It is therefore suitable for modelling interest
rates, stochastic volatilities and default intensities \cite{IGBMapproximation}. From a mathematical viewpoint, IGBM is
one of the simplest SDEs that has no known method of exact simulation \cite{IGBMapplication}.
By incorporating the ideas provided by the main theorem into the log-ODE method,
we will produce a state-of-the-art numerical approximation of IGBM. Although the vector
fields for IGBM are not bounded, our numerical evidence indicates that the
method has a strong convergence rate of $O(h^{\frac{3}{2}})$ and a weak convergence rate of $O(h^{2})$.

\subsection{Notation}
Below is some of the notation that is used throughout the paper.\vspace{-0.75mm}
\begin{center}
\renewcommand{\arraystretch}{1.1}
\begin{longtable}{cp{9.6cm}l}
\toprule
Symbol & Meaning & Page \vspace{0.5mm}\\
\toprule
$W$ & a standard real-valued Brownian motion. & \hspace{2mm} 2\vspace{0.5mm}\\
$B$ & a standard real-valued Brownian bridge on $[\hspace{0.1mm}0\hspace{0.05mm},1\hspace{0.1mm}]$. & \hspace{2mm} 5\vspace{1.5mm}\\
$\mu$ & a Borel measure on $[\hspace{0.1mm}0,1]$ defined by a singular weight function. & \hspace{2mm} 5\vspace{1.5mm}\\
&\begin{math}\begin{aligned}
\hspace{30mm}\mu(a,b) = \int_{a}^{b}\frac{1}{x(1-x)}\,dx,
\end{aligned}\end{math} & \vspace{1.5mm} \\
& for all open intervals $(a,b)\subset [\hspace{0.1mm}0,1]$. &\vspace{1.5mm}\\
$\left\{e_{k}\right\}_{k\geq 1}$ & a family of Jacobi-like polynomials with $\deg\left(e_{k}\right) = k + 1$ that
are orthogonal with respect to weight function $w(x) :=\hspace{0.25mm} \frac{1}{x(1-x)}\hspace{0.25mm}.$&  \hspace{2mm} 5 \vspace{1mm} \\
$I_{k}$ & a time integral of $B$ times the polynomial $e_{k}(t)\,w(t)$ over $[\hspace{0.1mm}0\hspace{0.05mm},1\hspace{0.1mm}]$, & \hspace{2mm} 5 \vspace{1.5mm}\\
&\begin{math}\begin{aligned}
\hspace{30mm}I_{k} = \int_{0}^{1}B_{t}\cdot\frac{e_{k}\left(t\right)}{t(1-t)}\,dt.
\end{aligned}\end{math} & \vspace{1.5mm} \\
$K_B$ & the covariance function of $B$, that is $K_B(s,t) = \min(s,t) - st$. & \hspace{2mm} 5\vspace{1.5mm}\\
\hspace{2mm}$P_{k}^{(\alpha, \beta)}$ & the $k$-th order $(\alpha, \beta)$-Jacobi polynomial on $[-1,1\hspace{0.1mm}]$ $\left(\alpha, \beta > -1\right)$.  & \hspace{1.25mm} 11\vspace{0.5mm}\\
\hspace{2mm}$Q_{k}$ & the $k$-th order Legendre polynomial on $[-1,1\hspace{0.1mm}]$, i.e. $Q_{k} = P_{k}^{(0, 0)}$. & \hspace{1.25mm} 13\vspace{0.5mm}\\
$y$ & a solution of the Stratonovich SDE on the finite interval $[\hspace{0.1mm}0\hspace{0.05mm},T\hspace{0.1mm}]$, & \hspace{1.25mm} 14 \vspace{1.5mm}\\
&\begin{math}\begin{aligned}
\hspace{26mm} dy_{t} & = f_{0}(y_{t})\,dt + f_{1}(y_{t})\circ dW_{t}\hspace{0.25mm},\\[3pt]
y_{0} & = \xi,
\end{aligned}\end{math} & \vspace{1.5mm} \\ 
& where $y, \xi\in\mathbb{R}^{e}$, and $f_{i}:\mathbb{R}^{e}\rightarrow\mathbb{R}^{e}$ denote smooth vector fields. & \\
& (It\^{o} SDEs will be defined on fixed intervals with the same form) & \vspace{0.5mm}\\
$[s,t]$ & a general closed subinterval of $[\hspace{0.15mm}0\hspace{0.05mm},T\hspace{0.1mm}]$, usually considered small. & \hspace{1.25mm} 14\vspace{0.5mm} \\
$h$ & the step size that a numerical method uses, typically $h = t- s$. & \hspace{1.25mm} 14\vspace{0.25mm}\\
$W_{s,t}$ & the increment of Brownian motion over $[\hspace{0.1mm}s\hspace{0.05mm},t\hspace{0.1mm}]$, $W_{s,t} := W_{t} - W_{s}\hspace{0.05mm}$. &  \hspace{1.25mm} 14 \vspace{0.5mm} \\
$\wideparen{W}$ & the Brownian parabola corresponding to $W$\hspace{-0.5mm} over some interval. &  \hspace{1.25mm} 15  \vspace{1.5mm} \\
&\begin{math}\begin{aligned}
\hspace{3.5mm}\wideparen{W}_{u} = W_{s} + \frac{u-s}{h}\,W_{s,t} + \frac{6\hspace{0.25mm}(u-s)(t-u)}{h^{2}}\,H_{s,t}\hspace{0.25mm},\hspace{2.9mm}\forall u\in [\hspace{0.1mm}s\hspace{0.05mm},t\hspace{0.1mm}].
\end{aligned}\end{math}& \vspace{1.5mm}\\
$Z$ & the Brownian arch corresponding to $W$ defined as $Z := W - \wideparen{W}$. & \hspace{1.25mm} 15 \vspace{0.5mm}\\
$H_{s,t}$ & the rescaled space-time L\'{e}vy area of Brownian motion on $[\hspace{0.1mm}s\hspace{0.05mm},t\hspace{0.1mm}]$, &  \hspace{1.25mm} 15 \vspace{1.5mm} \\
&\begin{math}\begin{aligned}
\hspace{22.5mm}H_{s,t} & = \frac{1}{h}\int_{s}^{t}W_{s,u} -\frac{u-s}{h}\,W_{s,t}\,du.
\end{aligned}\end{math} & \vspace{1.5mm} \\
$L_{s,t}$ & the space-space-time L\'{e}vy area of Brownian motion over $[\hspace{0.1mm}s\hspace{0.05mm},t\hspace{0.1mm}]$, & \hspace{1.45mm} 17 \vspace{1.5mm}\\
&\begin{math}\begin{aligned}
L_{s,t} & = 
\frac{1}{6}\hspace{-0.3mm}\left(\int_{s}^{t}\hspace{-1mm}\int_{s}^{u}\hspace{-1.5mm}\int_{s}^{v}\hspace{-1mm}\circ\,dW_{r}\hspace{-0.25mm}\circ dW_{v}\, du 
 - 2\hspace{-0.3mm}\int_{s}^{t}\hspace{-1mm}\int_{s}^{u}\hspace{-1.5mm}\int_{s}^{v}\hspace{-1mm}\circ\,dW_{r}\, dv\circ dW_{u}\right. \\
&\hspace{10mm}\left. + \int_{s}^{t}\hspace{-1mm}\int_{s}^{u}\hspace{-1.5mm}\int_{s}^{v}\hspace{-1mm}\,dr \circ dW_{v}\hspace{-0.25mm}\circ dW_{u}\right)\hspace{-0.5mm},
\end{aligned}\end{math}\vspace{1.5mm}\\
$Y$ & an approximation for the true solution $y$ of a Stratonovich SDE. & \hspace{1.25mm} 19 \vspace{0.5mm}\\
$[\,\cdot\hspace{0.5mm}, \cdot\,]$ & the standard Lie bracket of  vector fields, $[f_{0}, f_{1}] = f_{1}^{\prime}\hspace{0.125mm}f_{0} - f_{0}^{\prime}\hspace{0.125mm}f_{1}$. & \hspace{1.25mm} 19
\end{longtable}
\end{center}\vspace{-5.75mm}

\section{Main result}

It was shown in \cite{Multiscalebm} that Brownian motion can be generated
using Alpert-Rokhlin multiwavelets (see \cite{WaveletTheory}). The mother functions that generate
this wavelet basis are supported on $[\hspace{0.1mm}0,1]$ and are defined using polynomials as follows:\medbreak
\begin{definition}[Alpert-Rokhlin wavelets]\label{ARWavelet}
For $q\geq 1$, define the $q$ functions
$\phi^{q,1}, \cdots, \phi^{q,q} : [\hspace{0.15mm}0,1\hspace{0.15mm}]\hspace{-0.3mm}\rightarrow \mathbb{R}$ as piecewise polynomials of degree $q-1$ with pieces on
$[\hspace{0.15mm}0,\frac{1}{2}\hspace{0.15mm}]$, $[\hspace{0.15mm}\frac{1}{2}, 1\hspace{0.15mm}]$ that satisfy the following conditions for all $p\in\left\{1,\,\cdots\,, q\hspace{0.15mm}\right\}$ and $t\in [\hspace{0.15mm}0, \frac{1}{2}\hspace{0.15mm})$\hspace{0.25mm}$:$
\begin{align}
\phi^{q,p}(t) & = (-1)^{q+p-1}\phi^{q,p}(1 - t),\label{momentcondition1}\\[2pt]
\int_{0}^{1}\phi^{q,p}(t)\phi^{q,r}(t)\,dt & = \delta_{qr}\hspace{0.25mm}, \hspace{3mm}\text{for}\hspace{2mm}1 \leq r \leq q, \label{momentcondition2}\\
\int_{0}^{1}t^{k}\phi^{q,p}(t)\,dt & = 0, \hspace{5.65mm}\text{for}\hspace{2.2mm}0 \leq k \leq q - 1. \label{momentcondition3}
\end{align}
The Alpert-Rokhlin multiwavelets of order $q$ can now be generated by translating and scaling the mother functions $\phi^{q,p}$.
\begin{align*}
\phi_{nk}^{q,p}(t) := \frac{1}{\sqrt{2^{n}}}\,\phi^{q,p}(2^{n}t - k),
\end{align*}
for $n\geq 0$ and $k\in\left\{0,\cdots, 2^{n}-1\right\}$.
\end{definition}\medbreak\noindent
Whilst our results will not be presented in terms of the above wavelets, we shall see
that the polynomials of interest are directly related to conditions (\ref{momentcondition1}), (\ref{momentcondition2}) and (\ref{momentcondition3}).
The main result of this paper gives an effective method for approximating sample
paths of Brownian motion by a class of Jacobi-like polynomials. The proof is based on
the interpretation of these polynomials as eigenfunctions of an integral
operator defined by the Brownian bridge covariance function\footnote{The Brownian bridge is the centered Gaussian process with covariance $K_{B}(s,t) = \min(s,t)-st$.}. These orthogonal polynomials,
which lie at the heart of this paper, will also help us interpret the geometrical features
that certain normally distributed iterated integrals encode about the Brownian path.\medbreak
\begin{theorem}[A polynomial Karhunen-Lo\`{e}ve theorem for the Brownian bridge]\label{waveletthm}
Let $B$ denote a Brownian bridge on $[\hspace{0.1mm}0,1]$ and consider the Borel measure $\mu$ given by
\begin{align*}
\mu(a,b) := \int_{a}^{b}\frac{1}{x(1-x)}\,dx, \hspace{3.5mm}\text{for all open intervals}\hspace{1mm} (a,b)\subset [\hspace{0.1mm}0,1].
\end{align*}\noindent
Then there exists a family of orthogonal polynomials $\{e_{k}\}_{k\geq 1}\hspace{-0.25mm}$ with $\deg\left(e_{k}\right)\hspace{-0.25mm} =\hspace{-0.25mm} k\hspace{-0.25mm} +\hspace{-0.25mm} 1$ and
\begin{align*}
\int_{0}^{1}e_{i}\,e_{j}\,d\mu = \delta_{ij}\hspace{0.125mm},
\end{align*}\noindent
with $\delta_{ij}$ denoting the Kronecker delta, such that $B$ admits the following representation
\begin{align}\label{polyapprox}
B = \sum_{k=1}^{\infty}I_{k}e_{k}\hspace{0.125mm},
\end{align}\noindent
where  $\{I_{k}\}$ is the collection of independent centered Gaussian random variables with
\begin{align}\label{polyintegrals}
I_{k} := \int_{0}^{1}B_{t}\cdot\frac{e_{k}\left(t\right)}{t(1-t)}\,dt,
\end{align}
and
\begin{align*}
\var(I_{k}) = \frac{1}{k(k+1)}\,.
\end{align*}
Furthermore, $\left\{e_{k}\right\}$ is an optimal orthonormal basis of $L^{2}([\hspace{0.1mm}0,1], \mu)$ for approximating $B$
by truncated series expansions with respect to the following weighted $L^{2}(\mathbb{P})$ norm
\begin{align*}
\left\|X\right\|_{L_{\mu}^{2}(\mathbb{P})} := \sqrt{\hspace{0.25mm}\mathbb{E}\hspace{-0.25mm}\left[\hspace{0.25mm}\int_{0}^{1}\left(X_{s}\right)^{2}\,d\mu(s)\right]},
\end{align*}
where $X$ is a square $\mu$-integrable process.
\end{theorem}\medbreak
\begin{proof}
Our argument is that of the Karhunen-Lo\`{e}ve theorem in general $L^{2}$ spaces.
Note that $B$ is a square $\mu$-integrable process as
\begin{align*}
\mathbb{E}\left[\hspace{0.25mm}\int_{0}^{1}(B_{s})^{2}\,d\mu(s)\right]
= \int_{0}^{1}\mathbb{E}\big[(B_{s})^{2}\big]\,d\mu(s)
= \int_{0}^{1}s(1-s)\cdot\frac{1}{s(1-s)}\,ds
= 1 < \infty.
\end{align*}
Let $K_{B}$ denote the covariance function for the standard Brownian bridge on $[\hspace{0.1mm}0, 1]$.
Since $K_{B}(s,t) = \min(s,t) - st$, it can be shown by direct calculation that $K_{B}$ satisfies
\begin{align*}
\left\|K_{B}\right\|_{L^{2}([\hspace{0.1mm}0,1]^{2},\,\mu^{2})}^{2}
= \int_{0}^{1}\int_{0}^{1}(\min(s,t) - st)^{2}\,d\mu(s)d\mu(t)
= \frac{1}{3}\pi^{2} - 3 < \infty.
\end{align*}
Hence, it follows that the integral operator $T_{K} : L^{2}([\hspace{0.1mm}0,1],\,\mu)\rightarrow L^{2}([\hspace{0.1mm}0,1],\,\mu)$ given by
\begin{align*}
(T_{K}f)(t) := \int_{0}^{1}K_{B}(s,t)f(s)\,d\mu(s),
\end{align*}
is well-defined and continuous. In addition, the variance function $k_{B}(x) := K_{B}(x,x)$ for $x\in[\hspace{0.1mm}0,1]$ is $\mu$-integrable as
\begin{align*}
\int_{0}^{1}|k_{B}(x)|\,d\mu(x)
= \int_{0}^{1}x(1-x)\cdot\frac{1}{x(1-x)}\,dx
= 1 < \infty. 
\end{align*}
Therefore, we can apply Mercer's theorem for kernels on general $L^{2}$ spaces (see \cite{Mercer}).
It then follows from Mercer's theorem that there exists an orthonormal set $\{e_{k}\}_{k\geq 1}$ of
$L^{2}([\hspace{0.1mm}0,1],\,\mu)$ consisting of eigenfunctions of $T_{K}$ such that the corresponding sequence of
eigenvalues $\{\lambda_{k}\}_{k\geq 1}$ is non-negative. Moreover, the eigenfunctions corresponding to
non-zero eigenvalues are continuous on $[\hspace{0.1mm}0,1]$ and the kernel $K_{B}$ has the representation
\begin{align}\label{mercer}
K_{B}(s,t) = \sum_{k=1}^{\infty}\lambda_{k} e_{k}(s)e_{k}(t),
\end{align}
where the series (\ref{mercer}) converges absolutely and uniformly on compact subsets of $[\hspace{0.1mm}0,1]$.\medbreak\noindent
In the next part of the proof, we will see that each $e_{k}$ is a polynomial of degree $k+1$.
As each $e_{k}$ is an eigenfunction of $T_{K}$, we have\vspace{0.5mm}
\begin{align}\label{mercerv2}
\int_{0}^{1}\frac{\min(s,t)-st}{s(1-s)}\,e_{k}(s)\,ds = \lambda_{k}e_{k}(t).
\end{align}
Since $e_{k}\in L^{2}([\hspace{0.1mm}0,1],\,\mu)$, it follows that $e_{k}(0) = 0$ and $e_{k}(1) = 0$ for each $k\geq 1$.
Therefore by using the Leibniz integral rule to twice differentiate both sides of (\ref{mercerv2})
and then multiplying by $t(1-t)$, we observe that $e_{k}$ satisfies the differential equation
\begin{align}\label{mercerv3}
t(1-t)\lambda_{k}e_{k}^{\prime\prime}(t) + e_{k}(t) = 0.
\end{align}
Since $e_{k} \neq 0$, we have that $\lambda_{k}\neq 0$. Differentiating the LHS of the ODE (\ref{mercerv3}) produces
\begin{align*}
t(1-t)\frac{d^{2}}{dt^{2}}(e_{k}^{\prime}) + (1-2t)\frac{d}{dt}(e_{k}^{\prime})+\frac{1}{\lambda_{k}}e_{k}^{\prime} = 0. 
\end{align*} 
For $x\in[-1,1]$, we define the function
\begin{align*}
y_{k}(x) := e_{k}^{\prime}\left(\frac{1}{2}(1+x)\right)\hspace{-0.5mm}.
\end{align*}
Thus $y_{k}$ satisfies the following differential equation\vspace{1mm}
\begin{align}\label{legendrede}
(1-x^{2})y_{k}^{\prime\prime}(x) - 2xy_{k}^{\prime}(x) + \frac{1}{\lambda_{k}}y_{k}(x) = 0.
\end{align}
Remarkably, this is the Legendre differential equation \cite{ClassicalTheory}. It then follows using
classical Sturm-Liouville theory that $\frac{1}{\lambda_{k}} = k(k+1)$ and $y_{k}$ is proportional to the $k$-th Legendre polynomial.
Therefore, the derivative $e_{k}^{\prime}$ will be a constant multiple of the $k$-th shifted Legendre polynomial and hence
each $e_{k}$ is a polynomial of degree $k+1$.\medbreak\noindent
We can now define the following integrals for $k\geq 1$,
\begin{align*}
I_{k} := \int_{0}^{1}B_{t}\cdot\frac{e_{k}\left(t\right)}{t(1-t)}\,dt.
\end{align*}
It follows from Fubini's theorem that
\begin{align*}
\mathbb{E}[I_{k}] & = 0,\\[3pt]
\mathbb{E}[I_{i}I_{j}]
& = \mathbb{E}\left[\hspace{0.25mm}\int_{0}^{1}\int_{0}^{1}B_{s}B_{t}\,e_{i}(s)\,e_{j}(t)\,d\mu(s)\,d\mu(t)\right]\\
& = \int_{0}^{1}\int_{0}^{1}\mathbb{E}[B_{s}B_{t}]\,e_{i}(s)\,e_{j}(t)\,d\mu(s)\,d\mu(t)\\
& = \int_{0}^{1}\,e_{j}(t)\left(\int_{0}^{1}K_{B}(s,t)\,e_{i}(s)\,d\mu(s)\right)\,d\mu(t)\\[3pt]
& = \lambda_{i}\delta_{ij}.
\end{align*}
Since each $I_{k}$ is defined by a linear functional on the same Gaussian process $B$, we see
from the above that $\{I_{k}\}$ is a collection of uncorrelated (and therefore independent)
Gaussian random variables with
\begin{align*}
\mathbb{E}[I_{k}] & = 0,\\[3pt]
\var(I_{k}) & = \frac{1}{k(k+1)}.
\end{align*}
Finally, the $L^{2}(\mathbb{P})$ convergence we require follows as
\begin{align*}
\mathbb{E}\left[\left(B_{t}-\sum_{k=1}^{N}I_{k}e_{k}(t)\right)^{2}\right]
& = k_{B}(t) + \mathbb{E}\left[\hspace{0.25mm}\sum_{i,j=1}^{N}I_{i}I_{j}\,e_{i}(t)\hspace{0.25mm}e_{j}(t)\right]
- 2\hspace{0.25mm}\mathbb{E}\left[\hspace{0.25mm}B_{t}\sum_{k=1}^{N}I_{k}e_{k}(t)\right]\\
& = k_{B}(t) + \sum_{i=1}^{N}\lambda_{i}e_{i}^{2}(t)
- 2\hspace{0.25mm}\mathbb{E}\left[\hspace{0.25mm}\sum_{i=1}^{N}\int_{0}^{1}B_{s}B_{t}\,e_{i}(s)\hspace{0.25mm}e_{i}(t)\,d\mu(s)\right]\\[2pt]
& = k_{B}(t) - \sum_{i=1}^{N}\lambda_{i}e_{i}^{2}(t),
\end{align*}
which converges to $0$ by Mercer's theorem (\ref{mercer}).\medbreak\noindent
All that remains is to prove optimality for the truncated series expansions of (\ref{polyapprox}).
Let $\left\{f_{k}\right\}_{k\geq 1}$ denote an orthonormal basis of $L^{2}([\hspace{0.1mm}0,1], \mu)$ such that
\begin{align*}
B = \sum_{k=1}^{\infty}J_{k}f_{k},\hspace{3mm}\text{where}\hspace{3mm}
J_{k}:= \int_{0}^{1}B_{t}\,f_{k}(t)\,d\mu(t),\hspace{3mm}\forall k\geq 1.
\end{align*}
For $n\geq 1$, we consider an error process associated with the above: $r_{n} := \sum\limits_{k=n+1}^{\infty}J_{k}f_{k}\,$.\medbreak\noindent
Then the square $L^{2}(\mathbb{P})$ norm of the $n$-th error process admits the following expansion,
\begin{align*}
\left\|r_{n}(t)\right\|_{L^{2}(\mathbb{P})}^{2}
& = \mathbb{E}\hspace{-0.25mm}\left[\hspace{0.25mm}\sum_{i=n+1}^{\infty}\sum_{j=n+1}^{\infty}J_{i}J_{j}f_{i}(t)f_{j}(t)\right]\\[2pt]
& = \sum_{i=n+1}^{\infty}\sum_{j=n+1}^{\infty}\mathbb{E}\hspace{-0.25mm}\left[\hspace{0.25mm}\int_{0}^{1}\int_{0}^{1}B_{s}B_{t}\,f_{i}(s)f_{j}(t)\,d\mu(s)\,d\mu(t)\right]\hspace{-0.25mm}f_{i}(t)f_{j}(t)\\[2pt]
& = \sum_{i=n+1}^{\infty}\sum_{j=n+1}^{\infty}\left(\int_{0}^{1}\int_{0}^{1}K_{B}(s,t)\,f_{i}(s)f_{j}(t)\,d\mu(s)\,d\mu(t)\right)f_{i}(t)f_{j}(t).
\end{align*}
Integrating the above with respect to $\mu$ and using the orthogonality of $\{f_{k}\}_{k\geq 1}$ gives
\begin{align*}
\left\|r_{n}\right\|_{L_{\mu}^{2}(\mathbb{P})}^{2} = \int_{0}^{1}\left\|r(t)\right\|_{L^{2}(\mathbb{P})}^{2}\,d\mu(t)
= \sum_{k=n+1}^{\infty}\int_{0}^{1}\int_{0}^{1}K_{B}(s,t)f_{k}(s)f_{k}(t)\,d\mu(s)\,d\mu(t).
\end{align*}
Note that any optimal orthonormal basis of $L^{2}([\hspace{0.1mm}0,1], \mu)$ solves the following problem:
\begin{align*}
\min_{f_{k}} \left\|r_{n}\right\|_{L_{\mu}^{2}(\mathbb{P})}^{2} \hspace{3mm}
\text{subject to}\hspace{3mm} \left\|f_{k}\right\|_{L^{2}([\hspace{0.1mm}0,1],\mu)} = 1.
\end{align*}
By introducing Lagrange multipliers $\nu_{k}$, we wish to find functions $\{f_{k}\}$ that minimize
\begin{align*}
E_{n}[\left\{f_{k}\right\}]:=\hspace{-1mm}\sum_{k=n+1}^{\infty}\int_{0}^{1}\int_{0}^{1}\hspace{-0.35mm}K_{B}(s,t)f_{k}(s)f_{k}(t)d\mu(s)d\mu(t)- \nu_{k}\left(\int_{0}^{1}\hspace{-0.5mm}\big(f_{k}(s)\big)^{2}d\mu(s) - 1\right)\hspace{-0.35mm}.
\end{align*}
We will now consider the following square integrable functions, defined for $s,t\in\left(0,1\right)$:
\begin{align*}
\tilde{f}_{k}(t) := f_{k}(t)\cdot \frac{1}{\sqrt{t(1-t)}}\,, \hspace{5mm}
\tilde{K}_{B}(s,t) := K_{B}(s,t)\cdot \frac{1}{\sqrt{s(1-s)}}\cdot\frac{1}{\sqrt{t(1-t)}}\,.
\end{align*}
Therefore it is enough to find a family of functions $\{\tilde{f}_{k}\}$ in $L^{2}([\hspace{0.1mm}0,1])$ which minimizes
\begin{align*}
\tilde{E}_{n}[\{\tilde{f}_{k}\}]
:= \sum_{k=n+1}^{\infty}\int_{0}^{1}\int_{0}^{1}\tilde{K}_{B}(s,t)\tilde{f}_{k}(s)\tilde{f}_{k}(t)\,ds\,dt
- \nu_{k}\left(\int_{0}^{1}\big(\tilde{f}_{k}(s)\big)^{2}\,ds - 1\right)\hspace{-0.25mm}.
\end{align*}
To find a minimizer, we set the functional derivative of $\tilde{E}_{n}$ with respect to $\tilde{f}_{k}$ to zero.
\begin{align*}
\frac{\partial \tilde{E}_{n}}{\partial \tilde{f}_{k}(t)}
= 2\int_{0}^{1}\tilde{K}_{B}(s,t)\tilde{f}_{k}(s)\,ds - 2\nu_{k}\tilde{f}_{k}(t) = 0.
\end{align*}
By using the definitions of $\tilde{f}_{k}$ and $\tilde{K}_{B}$, it is trivial to show the above is equivalent to
\begin{align*}
\int_{0}^{1}K_{B}(s,t)f_{k}(s)\,d\mu(s) = \nu_{k}f_{k}(t),
\end{align*}
which is satisfied if and only if $f_{k}$ are eigenfunctions of $T_{K}$.
\end{proof}\medbreak\noindent
This result can naturally be extended to express Brownian motion using polynomials.\medbreak
\begin{theorem}\label{bmversion}
If $W$ is a standard Brownian motion and $B$ is the associated bridge process on $[\hspace{0.1mm}0,1]$,
then by Theorem \ref{waveletthm}, we have the below representation of $W$:
\begin{align}\label{secpolyapprox}
W = W_{1}e_{0} + \sum_{k=1}^{\infty}I_{k}e_{k}\hspace{0.125mm},
\end{align}
where $e_{0}(t) := t\hspace{0.25mm}$ for $\hspace{0.25mm}t\in[\hspace{0.1mm}0,1]$, and the random variables $\{I_{k}\}$ are independent of $W_{1}$.\vspace{-4mm}
\begin{figure}[h]\label{secparaboladiagram}
\centering
\includegraphics[width=0.975\textwidth]{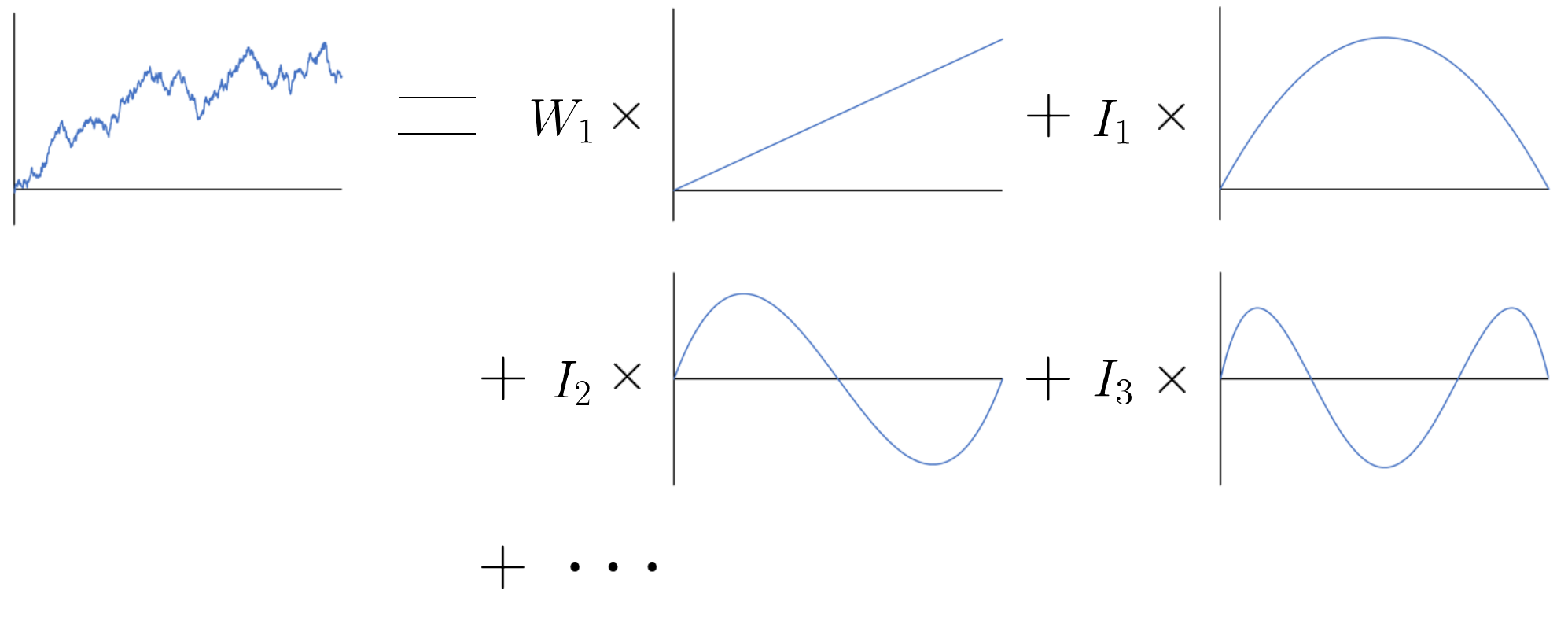}\vspace{-1.5mm}
\caption{Brownian motion can be expressed as a sum of polynomials with independent weights.
Moreover, these polynomials are orthogonal and capture different time integrals of the original path.}
\end{figure}
\end{theorem}

\noindent
In the rest of this section, we shall study the key objects introduced in Theorem \ref{waveletthm}.
Since each orthogonal polynomial lies in $L^{2}([\hspace{0.1mm}0,1],\,\mu)$, it must have roots at $0$ and $1$.
Therefore $e_{k}\cdot\frac{1}{t(1-t)}$ is itself a polynomial but with degree $k-1$, and one can repeatedly
apply the integration by parts formula to the stochastic integrals $\{I_{k}\}$ defined by ($\ref{polyintegrals}$).
This enables us to express each $I_{k}$ in terms of iterated integrals of Brownian motion.
Moreover, as $e_{k}\cdot\frac{1}{t(1-t)}$ has precisely $k-2$ non-zero derivatives, the highest order
iterated integral that is required to fully describe $I_{k}$ is $\int_{0<s_{1}<\cdots<s_{k}<1}B_{s_{1}}\,ds_{1}\,\cdots\,ds_{k}\hspace{0.125mm}$.\medbreak\noindent
So by applying the integration by parts formula as above, we can construct a lower triangular $n\times n$ matrix $M_{n}$ with non-zero diagonal entries that characterizes the relationship between $\{I_{k}\}_{1\leq k\leq n}$ and a set of $n$ iterated integrals of Brownian motion.\medbreak\noindent
Hence, for $n\geq 1$, we can express the $n$ independent Gaussian integrals $\{I_{k}\}_{1\leq k\leq n}$ as
\begin{align}\label{itintrelate}
\begin{pmatrix}
I_{1} \\ \vdots \\ I_{n}
\end{pmatrix} 
= M_{n}
\begin{pmatrix}
\int_{0<s_{1}<1}B_{s_{1}}\,ds_{1} \\ \vdots \\
\,\int_{0<s_{1}<\cdots<s_{n}<1}B_{s_{1}}\,ds_{1}\,\cdots\,ds_{n}\,
\end{pmatrix}.
\end{align}
Since $M_{n}$ is an invertible matrix, it follows that the column vectors appearing in (\ref{itintrelate})
both encode the same information about the Brownian bridge. This enables
us to establish a connection between Brownian motion, iterated integrals and polynomials.\medbreak
\begin{theorem}\label{polyapproxthm}
Consider\hspace{0.25mm} the\hspace{0.25mm} below\hspace{0.25mm} conditional\hspace{0.25mm} expectation\hspace{0.25mm} of\hspace{0.25mm} Brownian\hspace{0.25mm} motion,
\begin{align}\label{unbiasedestimate}
\hspace{-0.85mm}W_{t}^{n} := \mathbb{E}\hspace{-0.25mm}\left[W_{t}\,\Big|\,W_{1}\hspace{0.25mm}, \int_{0<s_{1}<1}W_{s_{1}}\,ds_{1}\hspace{0.25mm},
\cdots, \int_{0<s_{1}<\cdots<s_{n-1}<1}W_{s_{1}}\,ds_{1}\cdots \,ds_{n-1}\right]\hspace{-0.5mm}.
\end{align}\medbreak\noindent
where $t\in[\hspace{0.1mm}0,1]$. Then $W^{n}$ is the unique polynomial of degree $n$ with a root at $0$ that matches the increment $W_{1}$ and $n-1$ iterated time integrals of the path $\,W$ given by:
\begin{align}\label{nminusoneintegrals}
\int_{0<s_{1}<1}W_{s_{1}}\,ds_{1}\hspace{0.25mm},
\cdots, \int_{0<s_{1}<\cdots<s_{n-1}<1}W_{s_{1}}\,ds_{1}\cdots \,ds_{n-1}\hspace{0.125mm}.
\end{align}
\end{theorem}\medbreak
\begin{proof}
It is a direct consequence of (\ref{itintrelate}) that $W_{t}^{n} = \mathbb{E}[W_{t}\,|\,W_{1}, I_{1}, \cdots, I_{n-1}]$.
Hence by (\ref{secpolyapprox}) and independence of the random variables $\{W_{1}, I_{1}, \cdots\}$, we have that
\begin{align}\label{brownianpoly}
W^{n} = W_{1}e_{0} + \sum_{k=1}^{n-1}I_{k}e_{k}\hspace{0.125mm}.
\end{align}
Thus $W^{n}$ is indeed a polynomial of degree $n$ with a root at $0$ and that matches the
increment of the Brownian path. Without loss of generality we can now assume $n\geq 2$.
All that remains is to argue $W^{n}$ matches the $n-1$ iterated integrals given in (\ref{nminusoneintegrals}).
Using the orthogonality of $\{e_{k}\}$, it follows directly from (\ref{brownianpoly}) that for $1 \leq k \leq n -1$:
\begin{align*}
I_{k} & = \int_{0}^{1}\Big(W_{t} - W_{1}e_{0}(t)\Big)\cdot\frac{e_{k}\left(t\right)}{t(1-t)}\,dt\\[2pt]
& = \int_{0}^{1}\left(W_{t}^{n} + \sum_{m=n}^{\infty}I_{m}e_{m}(t) - W_{1}e_{0}(t)\right)\cdot\frac{e_{k}\left(t\right)}{t(1-t)}\,dt\\[2pt]
& = \int_{0}^{1}\Big(W_{t}^{n} - W_{1}e_{0}(t)\Big)\cdot\frac{e_{k}\left(t\right)}{t(1-t)}\,dt
+ \underbrace{\sum_{m=n}^{\infty}I_{m}\int_{0}^{1}\frac{e_{m}\left(t\right)e_{k}\left(t\right)}{t(1-t)}\,dt}_{=\,0.}\hspace{-0.3mm}.
\end{align*}
Hence $W^{n}$ matches the integrals of Brownian motion against polynomials with
degree
at most $n-1$. By the same argument used in the derivation of (\ref{itintrelate}), it follows that
$W^n$ matches the various iterated time integrals given in the statement of the theorem.
The uniqueness of $W^n$ is now a consequence of having $n+1$ different constraints.
\end{proof}\medbreak
\subsection{Properties of orthogonal polynomials} Although Theorem \ref{waveletthm} and
Theorem \ref{bmversion} are interesting results from a theoretical point of view, both lack an
explicit construction of the polynomials $\left\{e_{k}\right\}$ that could be implemented in practice.
On the other hand, it was shown that the defining eigenfunction property of each $e_{k}$
implies that its derivative $e_{k}^{\prime}$ is proportional to the $k$-th shifted Legendre polynomial.
Hence the family $\left\{e_{k}\right\}$ is the (normalized) shifted $(\alpha, \beta)$-Jacobi polynomials but with
$\alpha = \beta = -1$. Since Jacobi polynomials are typically studied with $\alpha, \beta > -1$, it is necessary to show there exists a well-defined limit when the parameters approach $-1$.\medbreak
\begin{figure}[h]\label{polynomialdiagams2}
\centering
\includegraphics[width=\textwidth]{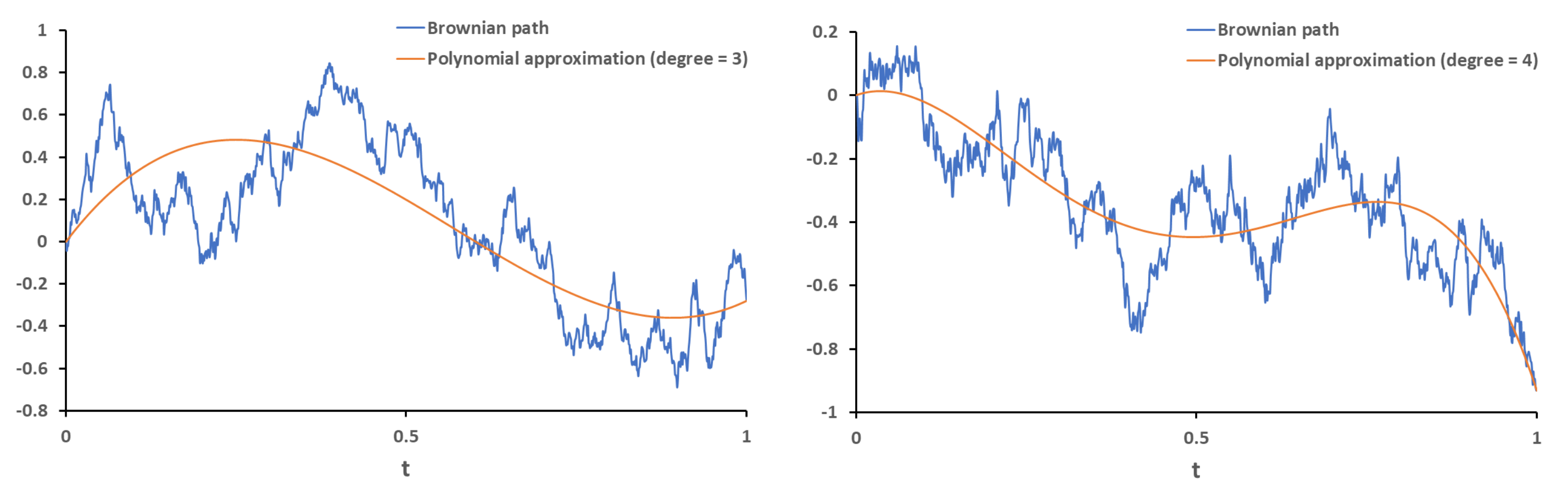}
\caption{Sample paths of Brownian motion with corresponding polynomial approximations.}
\end{figure}
\begin{definition}\label{jacobidef} For $k\geq 2$, the $k$-th degree $(\text{-}1,\text{-}1)$-Jacobi polynomial $P_{k}^{(\text{-}1, \text{-}1)}$ is
\begin{align*}
P_{k}^{(\text{-}1, \text{-}1)} := \lim_{\alpha, \beta \rightarrow -1^{+}}P_{k}^{(\alpha, \beta)}.
\end{align*}
\end{definition}\noindent
Naturally, for this definition to be unambiguous, we will require the following lemma.
\medbreak
\begin{lemma}\label{jacobicheck}
Let $P_{k}^{(\alpha, \beta)}$ denote the $k$-th degree $(\alpha, \beta)$-Jacobi polynomial on $[-1,1]$.
Then for $k\geq 2$, there exists a real-valued polynomial $P_{k}$ such that
$\big\|P_{k} - P_{k}^{(\alpha, \beta)}\big\|_{\infty} \rightarrow 0$ as $\alpha, \beta \rightarrow -1^{+}$.
\end{lemma}\medbreak
\begin{proof}
Below is an identity for Jacobi polynomials, with $\alpha, \beta > - 1$, given in \cite{JacobiPoly}.
\begin{align}\label{jacobifact}
P_{k}^{(\alpha, \beta)}(x) = \frac{k+\alpha+\beta+1}{2}\int_{-1}^{x}P_{k-1}^{(\alpha + 1, \beta + 1)}(u)\,du,
\hspace{5mm}\text{for all}\hspace{2mm} k \geq 2.
\end{align}
Therefore, we shall define the $k$-th degree polynomial $P_{k}$ over the interval $[-1,1]$ by
\begin{align}\label{jacobifact2}
P_{k}(x) := \frac{k-1}{2}\int_{-1}^{x}P_{k-1}^{(0, 0)}(u)\,du, \hspace{5mm}\text{for all}\hspace{2mm} k \geq 2.
\end{align}
It is straightforward to verify that $\lim_{\hspace{0.25mm}\alpha, \beta\, \rightarrow\, 0}\|\hspace{0.5mm}P_{n}^{(\alpha, \beta)} -\hspace{0.25mm} P_{n}^{(0, 0)}\hspace{0.25mm}\|_{\infty} = 0$ for $n\in\{0, 1\}$.
So by induction and the recurrence relation for Jacobi polynomials (see \cite{JacobiPoly}), we have:
\begin{align*}
\big\|\hspace{0.5mm}P_{n}^{(\alpha, \beta)} -\hspace{0.25mm} P_{n}^{(0, 0)}\hspace{0.25mm}\big\|_{\infty} \rightarrow 0\hspace{2.5mm}\text{as}\,\,\,\alpha, \beta\, \rightarrow\, 0,
\end{align*}
for all $n\geq 0$. Hence by the dominated convergence theorem with (\ref{jacobifact}) and (\ref{jacobifact2}), it follows that $P_{k}^{(\alpha, \beta)}$ will converge pointwise to $P_{k}$ as $\alpha, \beta\, \rightarrow\, -1$ for each $k\geq 1$. Finally, the result follows as $P_{k}^{(\alpha, \beta)}$ and $P_{k}$ are always polynomials with degree $k$.
\end{proof}\medbreak\noindent
Using the above definition for $(\text{-}1,\text{-}1)$-Jacobi polynomials, we can give an
explicit formula for the orthonormal polynomials $\{e_{k}\}_{k\geq 1}$ appearing in Theorems \ref{waveletthm} and \ref{bmversion}.\medbreak
\begin{theorem}\label{explicitformulathm}
Suppose each $e_{k}$ has a positive leading coefficient. Then for $k\geq 1$,
\begin{align}\label{explicitformula}
e_{k}(t) & = \frac{1}{k}\sqrt{k(k+1)(2k+1)}\,P_{k+1}^{(\text{-}1, \text{-}1)}(2t - 1),
\hspace{4mm}\forall t\in[\hspace{0.1mm}0,1].
\end{align}
\end{theorem}
\begin{proof}
The following identity for $(\alpha, \beta)$-Jacobi polynomials is stated in \cite{JacobiPoly}:
\begin{align*}
\int_{-1}^{1}(1-x)^{\alpha}(1+x)^{\beta}\big(P_{n}^{(\alpha,\beta)}(x)\big)^{2}\,dx
= \frac{2^{\alpha + \beta + 1}}{2n+\alpha+\beta+1}\frac{\Gamma(n+\alpha+1)\,\Gamma(n+\beta+1)}{\Gamma(n+\alpha+\beta+1)\,n!}\,,
\end{align*}
for $n\geq 1$ and $\alpha, \beta > -1$. Applying the change of variables, $t := \frac{1}{2}(x+1)$, we have
\begin{align*}
\int_{0}^{1}t^{\beta}(1-t)^{\alpha}\big(P_{n}^{(\alpha,\beta)}(2t-1)\big)^{2}dt
= \frac{1}{2n+\alpha+\beta+1}\frac{\Gamma(n+\alpha+1)\,\Gamma(n+\beta+1)}{\Gamma(n+\alpha+\beta+1)\,n!}\,,
\end{align*}
for $n\geq 1$ and $\alpha, \beta > -1$. By definition \ref{jacobidef}, taking the limit $\alpha, \beta \rightarrow -1^{+}$ will yield
\begin{align*}
\int_{0}^{1}\frac{1}{t(1-t)}\left(P_{n}^{(\text{-}1,\text{-}1)}(2t-1)\right)^{2}dt
& = \frac{1}{2n-1}\frac{(n-1)!\,(n-1)!}{n!\,(n-2)!}\\
& = \frac{1}{2n-1}\frac{n-1}{n}\,,\hspace{5mm}\text{for all}\hspace{2mm} n\geq 2.
\end{align*}
Therefore by setting $k := n - 1$, we have
\begin{align*}
\int_{0}^{1}\frac{1}{t(1-t)}\left(\frac{1}{k}\sqrt{k(k+1)(2k+1)}\,P_{n}^{(\text{-}1,\text{-}1)}(2t-1)\right)^{2}\, dt
= 1, \hspace{5mm}\text{for all}\hspace{1.25mm} k \geq 1.
\end{align*}
Recall that $e_{k}^{\prime}$ is proportional to the $k$-th shifted Legendre polynomial $P_{k}^{(0,0)}(2t-1)$.
Similarly, we saw in the proof of Lemma \ref{jacobicheck} that the derivative of $P_{k+1}^{(\text{-}1,\text{-}1)}$ is $\frac{k}{2}\hspace{0.25mm}P_{k}^{(0,0)}$.
As $e_{k}$ and $P_{k+1}^{(\text{-}1,\text{-}1)}$ are zero at their respective endpoints, we have that each $e_{k}$ must be proportional to $P_{k+1}^{(\text{-}1,\text{-}1)}(2t-1)$. The result now follows from the above calculations.
\end{proof}\medbreak\noindent
Having identified an explicit formula for the eigenfunctions $\{e_k\}$ in (\ref{explicitformula}), we shall now describe two methodologies for computing the
Jacobi-like polynomials $\big\{P_{k}^{(\text{-}1,\text{-}1)}\big\}$.\smallbreak\noindent
The first approach is to use the three-term recurrence relation in the theorem below.\medbreak
\begin{theorem}[Recurrence relation for $(\text{-}1,\text{-}1)$-Jacobi polynomials]
For $n\geq 2$\hspace{0.25mm},
\begin{align}\label{jacobirecurrence}
n(n+2)P_{n+2}^{(\text{-}1,\text{-}1)}(x) & = (n+1)(2n+1)\hspace{0.25mm}xP_{n+1}^{(\text{-}1,\text{-}1)}(x) - n(n+1)P_{n}^{(\text{-}1,\text{-}1)}(x)\hspace{0.25mm},
\end{align}
where the initial polynomials are given by
\begin{align*}
P_{2}^{(\text{-}1,\text{-}1)}(x) & = \frac{1}{4}\hspace{0.25mm}(x-1)(x+1)\hspace{0.25mm},\\[3pt]
P_{3}^{(\text{-}1,\text{-}1)}(x) & = \frac{1}{2}\hspace{0.25mm}x(x-1)(x+1)\hspace{0.25mm}.
\end{align*}
\end{theorem}\vspace{-1.5mm}
\begin{proof}
The below recurrence relation for Jacobi polynomials is presented in \cite{JacobiPoly},
\begin{align*}
& 2(k+1)(k+\alpha+\beta + 1)(2k+\alpha+\beta)P_{k+1}^{(\alpha,\beta)}(x)\\[3pt]
&\hspace{5mm} = (2k + \alpha + \beta + 1)\hspace{-0.35mm}\left((2k + \alpha + \beta)(2k + \alpha + \beta + 2)x + \alpha^{2} - \beta^{2}\right)P_{k}^{(\alpha,\beta)}(x)\\[3pt]
&\hspace{8.5mm} - 2(k+\alpha)(k+\beta)(2k+\alpha+\beta+2)\,P_{k-1}^{(\alpha,\beta)}(x)\hspace{0.25mm},
\end{align*}
for $k\geq 1$ and $\alpha, \beta > -1$. By definition \ref{jacobidef}, it is possible to take the limit $\alpha,\beta\rightarrow 1^{+}$
provided that $k\geq 3$. Therefore, taking this limit and setting $n = k - 1 \geq 2$ produces
\begin{align*}
4n^{2}(n+2)P_{n+2}^{(\text{-}1,\text{-}1)}(x) = 4n(n + 1)(2n + 1)xP_{n+1}^{(\text{-}1,\text{-}1)}(x) - 4n^{2}(n+1)\,P_{n}^{(\text{-}1,\text{-}1)}(x)\hspace{0.25mm},
\end{align*}
for $n\geq 2$. Dividing the above by $4n$ gives the required recurrence relation (\ref{jacobirecurrence}).
Finally the below formula, stated in \cite{JacobiPoly}, can be used to compute $P_{2}^{(\text{-}1,\text{-}1)}$ and $P_{3}^{(\text{-}1,\text{-}1)}$:
\begin{align*}
P_{n}^{(\alpha, \beta)}(x) = \frac{(-1)^{n}}{2^{n}n!}\left(1-x\right)^{-\alpha}\left(1+x\right)^{-\beta}\frac{d^{n}}{dx^{n}}\hspace{-0.25mm}\left(\left(1-x\right)^{\alpha + n}\left(1+x\right)^{\beta + n}\right),\hspace{3mm}\text{for}\,\, n\geq 2\hspace{0.25mm}.
\end{align*}
As before we take $\alpha,\beta\rightarrow 1^{+}$ in the above to obtain an explicit formula for $P_{n}^{(\text{-}1,\text{-}1)}$.
\end{proof}\smallbreak\noindent
In addition to computing these polynomials via a recurrence relation, it is also
possible to represent each $P_{n}^{(\text{-}1,\text{-}1)}$ as the difference of two (rescaled) Legendre polynomials.
Since the Legendre polynomials already have efficient implementations in the majority
of high-level programming languages, this second approach is particularly appealing.\medbreak

\begin{theorem}[Relationship between the Jacobi-like and Legendre polynomials]
For $n\geq 1$, we have
\begin{align*}
P_{n+1}^{(\text{-}1,\text{-}1)}(x) = \frac{n}{4n+2}\hspace{0.25mm}\big(Q_{n+1}(x) - Q_{n-1}(x)\big),
\end{align*}
where $Q_{k}$ denotes the $k$-th degree Legendre polynomial defined on $[-1,1]$.
\end{theorem}\medbreak
\begin{proof}
Recall that $\frac{d}{dx}\big(P^{(\text{-}1,\text{-}1)}_{n+1}\hspace{0.25mm}\big) = \frac{n}{2}P^{(0,0)}_{n}$ for $n\geq 1$, where $P^{(0,0)}_{n} (= Q_{n})$ is the $n$-th degree Legendre polynomial. Therefore differentiating both sides of (\ref{jacobirecurrence}) yields
\begin{align*}
\frac{1}{2}\hspace{0.25mm}n(n+1)(n+2)\hspace{0.25mm}Q_{n+1}(x) & = (n+1)(2n+1)P_{n+1}^{(\text{-}1,\text{-}1)}(x) + \frac{1}{2}\hspace{0.25mm}n(n+1)(2n+1)\hspace{0.25mm}x\hspace{0.25mm}Q_{n}(x) \\
&\hspace{5mm} - \frac{1}{2}\hspace{0.25mm}n(n-1)(n+1)\hspace{0.25mm}Q_{n-1}(x)\hspace{0.25mm}.
\end{align*}
Hence by simplifying and rearranging the above, we have that for $n\geq 1$,
\begin{align*}
(2n+1)P_{n+1}^{(\text{-}1,\text{-}1)}(x) & =  \frac{1}{2}\hspace{0.25mm}n\hspace{0.25mm}\big(Q_{n+1}(x) - Q_{n-1}(x)\big)\\
&\hspace{5mm} + \frac{1}{2}\hspace{0.25mm}n\hspace{0.25mm}\big((n+1)\hspace{0.25mm}Q_{n+1}(x) - (2n+1)\hspace{0.25mm}x\hspace{0.25mm}Q_{n}(x)  + n\hspace{0.25mm}Q_{n-1}(x)\big).
\end{align*}
We see the last term is zero by a recurrence relation for Legendre polynomials \cite{ClassicalTheory}.
\end{proof}\medbreak\noindent
In addition to viewing the polynomials $\{e_{k}\}$ as orthogonal with respect to the
weight function $w(x) := \frac{1}{x(1-x)}$\hspace{0.25mm}, we can characterize them via their iterated time integrals.
In particular, for $1\leq k\leq n -1$\hspace{0.25mm}, it follows from the integration by parts formula that
\begin{align*}
\hspace{10mm}\int_{0<s_{1}<\cdots<s_{k}<1}e_{n}(s_{1})\,ds_{1}\cdots \,ds_{k}
& = \int_{0}^{1}e_{n}(s)\,d\hspace{-0.25mm}\left(\frac{1}{k!}s^{k}\right)\\
& = \frac{1}{(k-1)!}\int_{0}^{1}s^{k-1}e_{n}(s)\,ds\\
& = -\frac{1}{k!}\int_{0}^{1}s^{k}e_{n}^{\prime}(s)\,ds\\[4pt]
& = 0\hspace{0.25mm}, \hspace{19.5mm}(\text{by the orthogonality of $e_{n}^{\prime}$}).
\end{align*}
Hence for $k\geq 1$, $e_{k}$ is a polynomial with degree $k+1$ that has roots at $0$ and $1$ as
well as $k-1$ trivial iterated integrals against time. By additionally specifying the $k$-th iterated
time integral, it is then possible to characterize the $k$-th polynomial $e_{k}$.\medbreak\noindent
To conclude this section, we will address the relationship between the orthogonal
Jacobi-like polynomials $\{e_{k}\}$ and the Alpert-Rokhlin wavelets given in definition \ref{ARWavelet}.
Since each $e_{k}^{\prime}$ is proportional to the $k$-th shifted Legendre polynomial, the family of
polynomials $\{e_{k}^{\prime}\}$ is orthogonal with respect to the standard $L^{2}([\hspace{0.1mm}0,1])$ inner product.
This orthogonality is exactly what is needed to satisfy the conditions (\ref{momentcondition2}) and (\ref{momentcondition3}).
Hence for any $q\geq 1$ there exists an Alpert-Rokhlin mother function of order $q$ that is a
piecewise polynomial where both pieces can be rescaled and translated to give $e_{q-1}^{\prime}$.

\section{Applications to SDEs}
Consider the Stratonovich SDE on the interval $[\hspace{0.1mm}0,T]$
\begin{align}\label{ogsde}
dy_{t} & = f_{0}(y_{t})\,dt + f_{1}(y_{t})\circ dW_{t}\hspace{0.125mm},\\
y_{0} & = \xi,\nonumber
\end{align}
where $\xi\in\mathbb{R}^{e}$ and $f_{i}$ denote bounded $C^{\infty}$ vector fields on $\mathbb{R}^{e}$ with bounded derivatives.
It then follows from the standard Picard iteration argument that there exists a unique
strong solution $y$ to (\ref{ogsde}). An important tool in the numerical analysis of this solution
is the stochastic Taylor expansion (see chapter 5 of \cite{KloePlat} for a comprehensive review).
For the purposes of this paper, we only require the following specific Taylor expansion.\medbreak
\begin{theorem}[High order Stratonovich-Taylor expansion]\label{stochtaylorthm}
Let $y$ denote the
unique strong solution to (\ref{ogsde}) and let $0\leq s \leq t$. Then $y_{t}$ can be expanded as follows:
\begin{align}\label{stochtaylorexp}
y_{t} & = y_{s} + f_{0}(y_{s})\hspace{0.25mm}h + f_{1}(y_{s})\hspace{0.25mm}W_{s,t} + \frac{1}{2}\hspace{0.25mm}f_{1}^{\prime}(y_{s})f_{1}(y_{s})\hspace{0.25mm}W_{s,t}^{2} + \frac{1}{2}\hspace{0.25mm}f_{0}^{\prime}(y_{s})f_{0}(y_{s})\hspace{0.25mm}h^{2}\\
&\hspace{3.5mm} + f_{0}^{\prime}(y_{s})f_{1}(y_{s})\int_{s}^{t}\int_{s}^{u}\circ\, dW_{v}\,du + f_{1}^{\prime}(y_{s})f_{0}(y_{s})\int_{s}^{t}\int_{s}^{u}\,dv\circ\, dW_{u}\nonumber \\
&\hspace{3.5mm} + \frac{1}{6}\hspace{0.25mm}\big(\hspace{0.25mm}f_{1}^{\prime}(y_{s})f_{1}^{\prime}(y_{s})f_{1}(y_{s}) + f_{1}^{\prime\prime}(y_{s})(f_{1}(y_{s}), f_{1}(y_{s}))\big)W_{s,t}^{3}\nonumber \\
&\hspace{3.5mm} + \big(\hspace{0.25mm}f_{0}^{\prime}(y_{s})f_{1}^{\prime}(y_{s})f_{1}(y_{s}) + f_{0}^{\prime\prime}(y_{s})(f_{1}(y_{s}), f_{1}(y_{s}))\big)\int_{s}^{t}\int_{s}^{u}\int_{s}^{v}\circ\,dW_{r}\circ\, dW_{v}\, du\nonumber \\
&\hspace{3.5mm} + \big(\hspace{0.25mm}f_{1}^{\prime}(y_{s})f_{0}^{\prime}(y_{s})f_{1}(y_{s}) + f_{1}^{\prime\prime}(y_{s})(f_{0}(y_{s}), f_{1}(y_{s}))\big)\int_{s}^{t}\int_{s}^{u}\int_{s}^{v}\circ\,dW_{r}\, dv\circ dW_{u}\nonumber \\
&\hspace{3.5mm} + \big(\hspace{0.25mm}f_{1}^{\prime}(y_{s})f_{1}^{\prime}(y_{s})f_{0}(y_{s}) + f_{1}^{\prime\prime}(y_{s})(f_{1}(y_{s}), f_{0}(y_{s}))\big)\int_{s}^{t}\int_{s}^{u}\int_{s}^{v}\,dr \circ dW_{v}\circ dW_{u}\nonumber \\
&\hspace{3.5mm} + \frac{1}{24}\hspace{0.25mm}\big(\hspace{0.25mm}f_{1}^{\prime}(y_{s})f_{1}^{\prime}(y_{s})f_{1}^{\prime}(y_{s})f_{1}(y_{s}) + f_{1}^{\prime}(y_{s})f_{1}^{\prime\prime}(y_{s})(f_{1}(y_{s}), f_{1}(y_{s}))\nonumber \\
&\hspace{14.6mm} +\hspace{0.25mm} 3\hspace{0.25mm}f_{1}^{\prime\prime}(y_{s})(f_{1}^{\prime}(y_{s})f_{1}(y_{s}),f_{1}(y_{s})) + f_{1}^{\prime\prime\prime}(y_{s})(f_{1}(y_{s}),f_{1}(y_{s}),f_{1}(y_{s})) \big)\hspace{0.25mm}W_{s,t}^{4}\nonumber \\[3pt]
&\hspace{3.5mm} + R(h, y_{s}),\nonumber
\end{align}
where $h:=t-s$ and the remainder term has the following uniform estimate for $h < 1$,
\begin{align}\label{highorderestimate}
\sup_{y_{s}\in\mathbb{R}^{e}}\left\|R(h, y_{s})\right\|_{L^{2}(\mathbb{P})} \leq C\, h^{\frac{5}{2}},
\end{align}
where the constant $C > 0$ depends only on the vector fields of the differential equation.
\end{theorem}\medbreak\noindent
From a numerical perspective, the most challenging terms presented in (\ref{stochtaylorexp}) are those
that involve non-trivial third order iterated integrals of Brownian motion and time.
Moreover, the most significant source of discretization error that high order numerical
methods will experience is generally due to approximating these stochastic integrals.
By representing Brownian motion as a (random) polynomial plus independent noise,
we shall derive a new optimal and unbiased estimator for these third order integrals.\medbreak
\begin{theorem}\label{polyplusnoise}
Let $W$ denote a standard real-valued Brownian motion on $[\hspace{0.1mm}0,1]$.
Let $W^{n}$ be the unique $n$-th degree random polynomial with a root at $0$ and satisfying
\begin{align*}
\int_{0}^{1}u^{k}\,dW_{u}^{n} & = \int_{0}^{1}u^{k}\,dW_{u}\hspace{0.125mm},\hspace{3mm}\text{for}\hspace{2mm}k = 0, 1, \cdots, n - 1.
\end{align*}
Then\hspace{0.25mm} $W =\hspace{0.25mm} W^{n} + Z^{n}$\hspace{0.25mm}, where\hspace{0.25mm} $Z^{n}$ is a centered Gaussian process independent of\hspace{0.5mm} $W^{n}$.\medbreak\noindent
Furthermore, $Z^{n}$ has the following covariance function:
\begin{align*}
\cov(Z^{n}_{s}, Z^{n}_{t}) = K_{B}(s,t) - \sum_{k=1}^{n-1}\lambda_{k}e_{k}(s)e_{k}(t), \hspace{3mm}\text{for}\hspace{2mm}s,t\in[\hspace{0.1mm}0,1],
\end{align*}
where $K_{B}$ denotes the standard Brownian bridge covariance function and $\{\lambda_{k}\}$, $\{e_{k}\}$
are the eigenvalues and eigenfunctions that were defined in the proof of Theorem \ref{waveletthm}.
\end{theorem}\medbreak
\begin{proof}
It follows from the integration by parts formula that $W^{n}$ matches the
increment and $n-1$ iterated time integrals of Brownian motion that appear in (\ref{nminusoneintegrals}).
Hence $W^{n}$ is also the polynomial defined in Theorem \ref{polyapproxthm} and $W = W^{n} + Z^{n}$ where\vspace{-0.25mm}
\begin{align*}
W^{n} & = W_{1}e_{0} + \sum_{k=1}^{n-1}I_{k}e_{k}\hspace{0.125mm},\\
Z^{n} & = \sum_{k=n}^{\infty}I_{k}e_{k}\hspace{0.125mm}.
\end{align*}
Then by Theorem \ref{waveletthm}, $Z^{n}$ is a centered Gaussian process that is independent of $W^{n}$.
In addition, the covariance function defining $Z^{n}$ can be directly computed as follows:
\begin{align*}
\cov(Z^{n}_{s}, Z^{n}_{t}) & = \cov\left(\sum_{i=n}^{\infty}I_{i}e_{i}(s), \sum_{j=n}^{\infty}I_{j}e_{j}(t)\right)\\
& = \sum_{k=n}^{\infty}\lambda_{k}e_{k}(s)e_{k}(t)\\
& = K_{B}(s,t) - \sum_{k=1}^{n-1}\lambda_{k}e_{k}(s)e_{k}(t),\hspace{3mm}\text{for}\hspace{2mm}s,t\in[\hspace{0.1mm}0,1].
\end{align*}
Note that the final line is achieved using the representation of $K_{B}$ given by (\ref{mercer}).
\end{proof}\medbreak\noindent
The above theorem has an interesting conclusion, namely that there exist unbiased
polynomial approximants of Brownian motion for which the error process can be
independently estimated in an $L^{2}(\mathbb{P})$ sense. In particular, this theorem already has
numerical applications in the case when $n=2$ and motivates the following definitions:\medbreak
\begin{definition}\label{brownianparabola}
The \textbf{standard Brownian parabola} $\,\wideparen{W}$ is the unique quadratic polynomial on $[\hspace{0.1mm}0,1]$ with a root at $0$ and satisfying
\begin{align*}
\wideparen{W}_{1} = W_{1}\hspace{0.125mm}, \hspace{5mm}\int_{0}^{1}\wideparen{W}_{u}\, du = \int_{0}^{1}W_{u}\,du.
\end{align*}
\end{definition}\vspace{-1mm}
\begin{definition}\label{brownianarch}
The \textbf{standard Brownian arch} $\,Z$ is the process $Z := W - \wideparen{W}$.
By Theorem \ref{polyplusnoise}, $Z$ is the centered Gaussian process on $[\hspace{0.1mm}0,1]$ with covariance function
\begin{align*}
K_{Z}\left(s,t\right) = \min(s,t) - st - 3st(1-s)(1-t), \hspace{3mm}\text{for}\hspace{2mm}s,t\in[\hspace{0.1mm}0,1].
\end{align*}
\end{definition}\vspace{-1mm}
\begin{definition}\label{spacetimelevy}
The rescaled \textbf{space-time L\'{e}vy area} of Brownian motion over
an interval $[s,t]$ with length $h$ encodes the signed area of the associated bridge process,
\begin{align*}
H_{s,t} := \frac{1}{h}\int_{s}^{t}W_{s,u} - \frac{u-s}{h}\,W_{s,t}\,du.
\end{align*}
\end{definition}\vspace{-0.5mm}
\begin{remark}
Since $e_{1}(t) = \sqrt{6}\,t(1-t)$, we have that $H_{0,1}$ corresponds to $\frac{\sqrt{6}}{6}I_{1}$ as defined in Theorem \ref{waveletthm}. Thus, $H_{s,t}\sim \mathcal{N}\hspace{-0.25mm}\left(0,\frac{1}{12}h\right)$ and $H_{s,t}$ is independent of $\hspace{0.25mm}W_{s,t}\hspace{0.25mm}$.
\end{remark}\medbreak\noindent
By applying the natural scaling of Brownian motion, one can define the Brownian
parabola and Brownian arch processes over any interval $[s,t]$ with finite size $h = t - s$.
Whilst the Brownian arch can be viewed in a similar light to the Brownian bridge,
there are clear qualitative and quantitative differences in their covariance functions.
In particular, the Brownian arch has less variance at its midpoint compared to most
points in $[s, t]$ $\big($by which we mean that $|\{u\in[s,t] : \var(Z_{u}) \leq \var(Z_{\frac{1}{2}(s+t)})\}| < \frac{1}{2}h\big)$.
This is in contrast to the Brownian bridge, which has most variance at its midpoint.
In fact, the Brownian parabola gives a relatively uniform estimate of the original path.
\vspace{-2.5mm}
\begin{figure}[h]\label{variancediagram}
\centering
\includegraphics[width=0.825\textwidth]{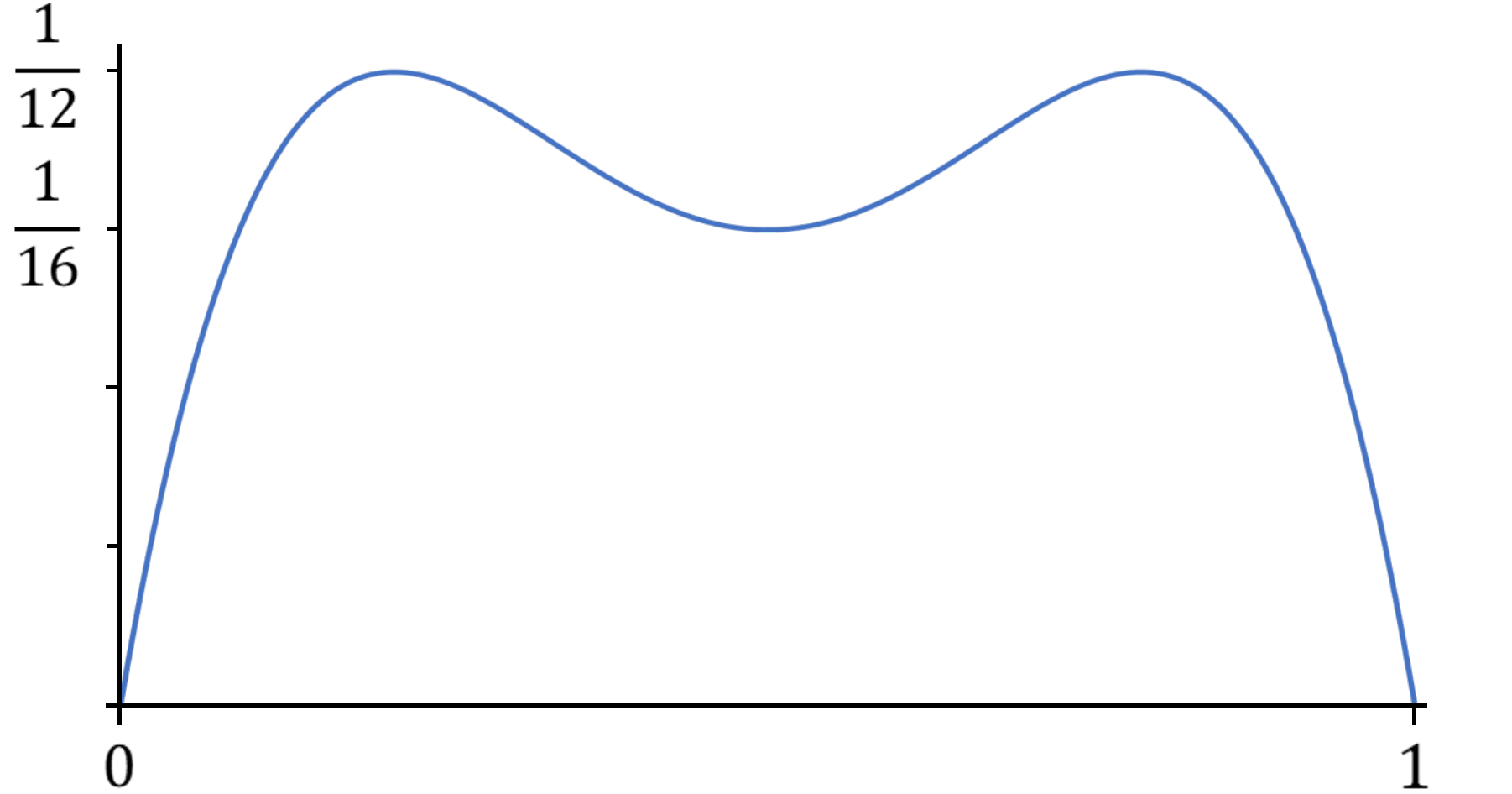}
\caption{Variance profile of the standard Brownian arch.}\vspace{-2.5mm}
\end{figure}\medbreak\noindent
Using these new definitions, we can study the high order integrals appearing in (\ref{stochtaylorexp}).\medbreak
\begin{theorem}[Conditional expectation of a non-trivial Brownian time integral]\vspace{-2.5mm}\label{firstcondthm}
\begin{align}\label{firstcondexp}
\mathbb{E}\hspace{-0.25mm}\left[\hspace{0.25mm}\int_{s}^{t}W_{s,u}^{2}\,du\, \Big|\, W_{s,t}\hspace{0.125mm}, H_{s,t}\right] = \frac{1}{3}hW_{s,t}^{2} + hW_{s,t}H_{s,t} + \frac{6}{5}hH_{s,t}^{2} + \frac{1}{15}h^{2}.
\end{align}
\end{theorem}
\begin{proof}
By the natural Brownian scaling it is enough to prove the result on $[\hspace{0.1mm}0,1]$.
Recall that $W = \wideparen{W} + Z$ where the parabola $\wideparen{W}$ is completely determined by $\left(W_{1}, H_{1}\right)$
and $Z$ is independent of $\left(W_{1}, H_{1}\right)$. This leads to a decomposition for the LHS of (\ref{firstcondexp}).
\begin{align*}
\mathbb{E}\hspace{-0.25mm}\left[\hspace{0.25mm}\int_{0}^{1}\hspace{-0.25mm}W_{u}^{2}\,du\, \Big|\, W_{1}, H_{1}\right]
& = \mathbb{E}\hspace{-0.25mm}\left[\hspace{0.25mm}\int_{0}^{1}\left(\hspace{0.25mm}\wideparen{W}_{u}+Z_{u}\right)^{2}\,du\, \Big|\, W_{1}, H_{1}\right]\\[2pt]
& = \mathbb{E}\hspace{-0.25mm}\left[\hspace{0.25mm}\int_{0}^{1}\wideparen{W}_{u}^{2}\,du + 2\int_{0}^{1}\wideparen{W}_{u}\hspace{0.25mm}Z_{u}\,du
+ \int_{0}^{1}Z_{u}^{2}\,du\,\Big|\, W_{1}, H_{1}\right]\\[2pt]
& = \int_{0}^{1}\wideparen{W}_{u}^{2}\,du + 2\int_{0}^{1}\wideparen{W}_{u}\,\mathbb{E}\left[Z_{u}\right]\,du
+ \int_{0}^{1}\mathbb{E}\hspace{-0.25mm}\left[Z_{u}^{2}\hspace{0.25mm}\right]\,du\\[3pt]
& = \int_{0}^{1}\hspace{-1mm}\big(uW_{1} + 6u(1-u)H_{1}\big)^{2}du +\hspace{-0.1mm}\int_{0}^{1}\hspace{-0.75mm}u - u^{2} - 3u^{2}(1-u)^{2}\,du.
\end{align*}
The result now follows by evaluating the above integrals.
\end{proof}\medbreak\noindent
The above theorem has practical applications for SDE simulation as $W_{s,t}$ and $H_{s,t}$ are
independent Gaussian random variables and can be easily generated or approximated.
That said, we should first discuss how the iterated integrals within (\ref{stochtaylorexp}) are connected.
\medbreak
\begin{definition}\label{spacespacetimelevy}
The \textbf{space-space-time L\'{e}vy area} of Brownian motion over
an interval $[s,t]$ is defined as\vspace{-0.5mm}
\begin{align*}
L_{s,t} & := \frac{1}{6}\left(\int_{s}^{t}\int_{s}^{u}\int_{s}^{v}\circ\,dW_{r}\circ\, dW_{v}\, du 
- 2\int_{s}^{t}\int_{s}^{u}\int_{s}^{v}\circ\,dW_{r}\, dv\circ dW_{u}\right. \\
&\hspace{11mm}\left. + \int_{s}^{t}\int_{s}^{u}\int_{s}^{v}\,dr \circ dW_{v}\circ dW_{u}\right).\nonumber
\end{align*}
\end{definition}\noindent
We can interpret $L_{s,t}$ as an area between the processes $\{W_{s,u}\}_{u\in[s,t]}$ and $\{H_{s,u}\}_{u\in[s,t]}$.
Moreover, rough path theory provides an algebraic structure (called the log-signature)
that relates $(W_{s,t}, H_{s,t}, L_{s,t})$ to the iterated integrals of space-time Brownian motion
and ultimately to SDE solutions via the log-ODE method (see \cite{LogSig} for an overview).
For our purposes, it is enough to give formulae relating these L\'{e}vy areas to integrals.\medbreak
\begin{theorem}\label{logsigrelation} Let $H_{s,t}$ and $L_{s,t}$ denote the L\'{e}vy areas of Brownian motion given
by definitions \ref{spacetimelevy} and \ref{spacespacetimelevy} respectively. Then the following integral relationships hold,
\begin{align*}
\int_{s}^{t}\int_{s}^{u}\circ\, dW_{v}\,du & = \frac{1}{2}hW_{s,t} + hH_{s,t}\hspace{0.125mm},\\
\int_{s}^{t}\int_{s}^{u}\,dv\circ\, dW_{u} & = \frac{1}{2}hW_{s,t} - hH_{s,t}\hspace{0.125mm},\\
\int_{s}^{t}\int_{s}^{u}\int_{s}^{v}\circ\,dW_{r}\circ\,dW_{v}\, du & = \frac{1}{6}hW_{s,t}^{2} + \frac{1}{2}hW_{s,t}H_{s,t} + L_{s,t}\hspace{0.125mm},\\
\int_{s}^{t}\int_{s}^{u}\int_{s}^{v}\circ\,dW_{r}\, dv\circ dW_{u} & = \frac{1}{6}hW_{s,t}^{2} - 2\hspace{0.125mm}L_{s,t}\hspace{0.125mm},\\
\int_{s}^{t}\int_{s}^{u}\int_{s}^{v}\,dr \circ dW{v}\circ dW_{u} & = \frac{1}{6}hW_{s,t}^{2} - \frac{1}{2}hW_{s,t}H_{s,t} + L_{s,t}\hspace{0.125mm}.
\end{align*}
\end{theorem}\vspace{-0.5mm}
\begin{proof} The result follows from numerous applications of integration by parts.
\end{proof}\medbreak\noindent
We can now present the new unbiased estimator for third order iterated integrals
of Brownian motion and time. The proposed estimator is fast to compute and the best
$L^{2}(\mathbb{P})$ approximation of these integrals that is measurable with respect to $\left(W_{s,t}, H_{s,t}\right)$.\medbreak
\begin{theorem}[Conditional moments of Brownian space-space-time L\'{e}vy area]\vspace{-3.5mm}\label{secondcondthm}
\begin{align}
\mathbb{E}\left[L_{s,t}\,\big|\,W_{s,t}\hspace{0.125mm}, H_{s,t}\right] & =  \frac{1}{30}\hspace{0.25mm}h^{2} + \frac{3}{5}\hspace{0.25mm}hH_{s,t}^{2}\,,\label{secondcondexp}\\[3pt]
\var\left(L_{s,t}\,\big|\,W_{s,t}\hspace{0.125mm}, H_{s,t}\right) & = \frac{11}{25200}\hspace{0.25mm}h^{4} + h^{3}\left(\frac{1}{720}\hspace{0.25mm}W_{s,t}^{2} + \frac{1}{700}\hspace{0.25mm}H_{s,t}^{2}\right)\hspace{-0.5mm}.\label{firstcondvar}
\end{align}
\end{theorem}\vspace{-1mm}
\begin{proof}
The expectation (\ref{secondcondexp}) is simply a consequence of Theorems \ref{firstcondthm} and \ref{logsigrelation}.
Without loss of generality, we will consider the above conditional variance on $[\hspace{0.1mm}0,1]$.
Since $\wideparen{W}$ is determined using the increment $W_{1}$ and space-time L\'{e}vy area $H_{1}$, we have
\begin{align*}
\var\left(\hspace{0.25mm}\int_{0}^{1}W_{u}^{2}\,du\, \Big|\, W_{1}, H_{1}\right) & = \var\left(\hspace{0.25mm}\int_{0}^{1}\wideparen{W}_{u}^{2}\,du + 2\int_{0}^{1}\wideparen{W}_{u}\hspace{0.25mm}Z_{u}\,du
+\hspace{-0.25mm}\int_{0}^{1}\hspace{-0.5mm}Z_{u}^{2}\,du\, \Big|\, W_{1}, H_{1}\right)\\[2pt]
& = 4\var\left(\hspace{0.25mm}\int_{0}^{1}\wideparen{W}_{u}\hspace{0.25mm}Z_{u}\,du\, \Big|\, W_{1}, H_{1}\right) + \var\left(\hspace{0.25mm}\int_{0}^{1}\hspace{-0.5mm}Z_{u}^{2}\,du\right)\\
&\hspace{7.5mm} + 4\cov\left(\hspace{0.25mm}\int_{0}^{1}\wideparen{W}_{u}\,Z_{u}\,du\hspace{0.25mm}, \int_{0}^{1}\hspace{-0.5mm}Z_{u}^{2}\,du\,\Big|\, W_{1}\hspace{0.125mm}, H_{1}\right)\hspace{-0.25mm}.
\end{align*}
Recall $Z = \sum\limits_{k=2}^{\infty}I_{k}\hspace{0.25mm}e_{k}$ where $\left\{I_{k}\right\}$ are independent centered Gaussian random variables.\medbreak\noindent
In particular, this means that $Z$ and $-Z$ have the same law. Therefore, we have that
\begin{align*}
\mathbb{E}\hspace{-0.25mm}\left[\hspace{0.25mm}\int_{0}^{1}\wideparen{W}_{u}\hspace{0.25mm}Z_{u}\,du\, \Big|\, W_{1}\hspace{0.125mm}, H_{1}\hspace{0.25mm}\right] =  0\hspace{0.25mm},
\end{align*}\vspace{-2mm}
\begin{align*}
\cov\left(\hspace{0.25mm}\int_{0}^{1}\wideparen{W}_{u}\hspace{0.25mm}Z_{u}\,du\hspace{0.25mm}, \int_{0}^{1}Z_{u}^{2}\,du\,\Big|\, W_{1}\hspace{0.125mm}, H_{1}\right) = 0\hspace{0.25mm}.
\end{align*}
The remaining two terms were resolved with assistance from Wolfram Mathematica.
\begin{align*}
&\var\left(\hspace{0.25mm}\int_{0}^{1}\wideparen{W}_{u}\hspace{0.25mm}Z_{u}\,du\, \Big|\, W_{1}\hspace{0.125mm}, H_{1}\right)\\
&\hspace{10mm} = \int_{0}^{1}\int_{0}^{1}\wideparen{W}_{u}\hspace{0.25mm}\wideparen{W}_{v}\,\mathbb{E}\left[Z_{u}\hspace{0.25mm}Z_{v}\,|\, W_{1}\hspace{0.125mm}, H_{1}\right]\,du\,dv\\
&\hspace{10mm} = \int_{0}^{1}\int_{0}^{1}\wideparen{W}_{u}\hspace{0.25mm}\wideparen{W}_{v}\big(\hspace{-0.25mm}\min(u,v) - uv - 3uv(1-u)(1-v)\big)\,du\,dv\\[1pt]
&\hspace{10mm} = 2\int_{0}^{1}\wideparen{W}_{v}\int_{0}^{v}u\hspace{0.25mm}\wideparen{W}_{u}\,du\,dv - \left(\int_{0}^{1}u\hspace{0.25mm}\wideparen{W}_{u}\,du\right)^{2} - 3\left(\hspace{0.25mm}\int_{0}^{1}u(1-u)\wideparen{W}_{u}\,du\right)^{2}\\[1pt]
&\hspace{10mm} = 2\left(\hspace{0.125mm}\frac{1}{15}\hspace{0.25mm}W_{1}^{2}+\frac{13}{60}\hspace{0.25mm}W_{1}H_{1}+\frac{13}{70}\hspace{0.25mm}H_{1}^{2}\right) - \left(\hspace{0.125mm}\frac{1}{3}\hspace{0.25mm}W_{1} + \frac{1}{2}\hspace{0.25mm}H_{1}\right)^{2} -3\left(\hspace{0.125mm}\frac{1}{12}\hspace{0.25mm}W_{1}+\frac{1}{5}\hspace{0.25mm}H_{1}\right)^{2}\\[3pt]
&\hspace{10mm} = \frac{1}{720}\hspace{0.25mm}W_{1}^{2} + \frac{1}{700}\hspace{0.25mm}H_{1}^{2}\hspace{0.25mm}.
\end{align*}\vspace{-3.5mm}
\begin{align*}
\var\left(\int_{0}^{1}Z_{u}^{2}\,du\right)
& = \mathbb{E}\hspace{-0.25mm}\left[\left(\hspace{0.25mm}\int_{0}^{1}Z_{u}^{2}\,du\right)^{2}\hspace{0.25mm}\right] - \left(\mathbb{E}\hspace{-0.25mm}\left[\hspace{0.25mm}\int_{0}^{1}Z_{u}^{2}\,du\right]\right)^{2}\\[1pt]
& = \int_{0}^{1}\int_{0}^{1}\mathbb{E}\hspace{-0.25mm}\left[Z_{u}^{2}\hspace{0.25mm}Z_{v}^{2}\right]\,du\,dv - \left(\mathbb{E}\left[\hspace{0.25mm}\int_{0}^{1}Z_{u}^{2}\,du\right]\right)^{2}\\[1pt]
& = \int_{0}^{1}\int_{0}^{1}\mathbb{E}\hspace{-0.25mm}\left[Z_{u}^{2}\hspace{0.25mm}\right]\mathbb{E}\hspace{-0.25mm}\left[Z_{v}^{2}\hspace{0.25mm}\right] + 2\hspace{0.25mm}\big(\hspace{0.25mm}\mathbb{E}[Z_{u}\hspace{0.25mm}Z_{v}]\hspace{0.25mm}\big)^{2}\,du\,dv - \left(\mathbb{E}\hspace{-0.25mm}\left[\hspace{0.25mm}\int_{0}^{1}Z_{u}^{2}\,du\right]\right)^{2}\\[3pt]
& = 4\int_{0}^{1}\int_{0}^{v}\big(u - uv - 3uv(1-u)(1-v)\big)^{2}\,du\,dv \\[3pt]
& = \frac{11}{6300}\hspace{0.25mm}.
\end{align*}
By Theorem \ref{logsigrelation}, the above gives an explicit formula for the conditional variance (\ref{firstcondvar}).
\begin{align*}
\var\left(L_{1}\,\big|\,W_{1}\hspace{0.125mm}, H_{1}\right) = \var\left(\frac{1}{2}\int_{0}^{1}W_{u}^{2}\,du\, \Big|\, W_{1}\hspace{0.125mm}, H_{1}\right) = \frac{11}{25200} + \frac{1}{720}\hspace{0.25mm}W_{1}^{2} + \frac{1}{700}\hspace{0.25mm}H_{1}^{2}\hspace{0.25mm}.
\end{align*}
By the natural Brownian scaling, the result on the interval $[s,t]$ directly follows.
\end{proof}\medbreak
\begin{remark}
The conditional variance (\ref{firstcondvar}) allows one to estimate local $L^{2}(\mathbb{P})$ errors for certain numerical methods and thus may be useful when choosing step sizes.
\end{remark}\medbreak\noindent
Therefore in order to propagate a numerical solution of (\ref{ogsde}) over an interval $[s,t]$,
one can generate $\left(W_{s,t}, H_{s,t}\right)$ exactly and then approximate $L_{s,t}$ using Theorem \ref{secondcondthm}.
However, there are many numerical methods that could be used to solve a given SDE.
\subsection{Examples of ODE methods} We will consider the following two methods:\medbreak
\begin{definition}[High order log-ODE method]\label{logodedef}
For a fixed number of steps $N$
we can construct a numerical solution $\left\{Y_{k}\right\}_{0\leq k \leq N}$ of (\ref{ogsde}) by setting $Y_{0} := \xi$ and for
each $k \in [0 \mathrel{{.}\,{.}}\nobreak N-1]$, defining $Y_{k+1}$ to be the solution at $u=1$ of the following ODE:
\begin{align}\label{logode}
\frac{dz}{du} & = f_{0}(z)\hspace{0.25mm}h + f_{1}(z)\hspace{0.25mm}W_{t_{k},t_{k+1}} + [f_{1}, f_{0}](z)\cdot hH_{t_{k},t_{k+1}}\\
&\hspace{14.5mm} + [f_{1}, [f_{1}, f_{0}]](z)\cdot
\mathbb{E}\big[\hspace{0.25mm}L_{t_{k},t_{k+1}}\, \big|\, W_{t_{k},t_{k+1}}\hspace{0.125mm}, H_{t_{k},t_{k+1}}\hspace{0.25mm}\big],\nonumber\\
z_{0} & = Y_{k}\hspace{0.125mm},\nonumber
\end{align}\noindent
where $h := \frac{T}{N}$, $t_{k} := kh$ and $[\,\cdot\hspace{0.5mm}, \cdot\,]$ denotes the standard Lie bracket of vector fields.
\end{definition}\medbreak
\begin{definition}[The parabola-ODE method]\label{parabolaodedef}
For a fixed number of steps $N$
we can construct a numerical solution $\left\{Y_{k}\right\}_{0\leq k \leq N}$ of (\ref{ogsde}) by setting $Y_{0} := \xi$ and for
each $k \in [0 \mathrel{{.}\,{.}}\nobreak N-1]$, defining $Y_{k+1}$ to be the solution at $u=1$ of the following ODE:
\begin{align}\label{paraode1}
\frac{dz}{du} & = f_{0}(z)\hspace{0.25mm}h + f_{1}(z)\hspace{0.25mm}\big(W_{t_{k},t_{k+1}} + (6 - 12u)\hspace{0.25mm}H_{t_{k},t_{k+1}}\big),\\
z_{0} & = Y_{k}\hspace{0.125mm},\nonumber
\end{align}
where\hspace{0.125mm} $h := \frac{T}{N}$ and\hspace{0.125mm} $t_{k} := kh$.
\end{definition}\medbreak\noindent
In both numerical methods the true solution $y$ at time $t_{k}$ can be approximated by $Y_{k}$.
Whilst there are different ways of interpolating between the successive approximations
$Y_{k}$ and $Y_{k+1}$, for this paper we will simply interpolate between such points linearly.
To analyse the above methods, we shall first note the key differences between them.
The first important distinction between the two methods is a purely practical one.
Although both of these methods involve computing a numerical solution of an ODE, the
parabola method does not require one to explicitly resolve vector field derivatives.
The second significant difference can be seen in the Taylor expansions of the methods.\medbreak
\begin{theorem}\label{localerrors}
Let $Y^{\text{log}}$ be the one-step approximation defined by the log-ODE method
on the interval $[s,t]$ with initial value $Y_{0}^{\text{log}} = y_{s}$. Then for sufficiently small $h$
\begin{align}\label{logodestep}
Y_{1}^{\text{log}} = y_{t} - [f_{1}, [f_{1}, f_{0}]]
(y_{s})\hspace{0.5mm}\big(\hspace{0.25mm}L_{s,t} - \mathbb{E}\big[\hspace{0.25mm}L_{s,t}\, \big|\, W_{s,t}\hspace{0.125mm}, H_{s,t}\hspace{0.25mm}\big]\big) + O(h^{\frac{5}{2}}).
\end{align}
Similarly, let $Y^{\text{para}}$ denote the one-step approximation given by the parabola-ODE method on
the interval $[s,t]$ with the same initial value. Then for sufficiently small $h$
\begin{align}\label{parabolaodestep}
Y_{1}^{\text{para}} = y_{t} - [f_{1}, [f_{1}, f_{0}]]
(y_{s})\left(L_{s,t} - \frac{3}{5}H_{s,t}^{2}\right) + O(h^{\frac{5}{2}}).
\end{align}
Note that $O(h^{\frac{5}{2}})$ denotes terms which can be estimated in an $L^{2}(\mathbb{P})$ sense as in (\ref{highorderestimate}).
\end{theorem}\medbreak
\begin{proof}
In order to derive (\ref{logodestep}), we must compute the Taylor expansion of (\ref{logode}).
Let $F$ denote the vector field defined in (\ref{logode}) that was constructed from $f_{0}$ and $f_{1}$.
Then $F$ is smooth, and it follows from the classical Taylor's theorem for ODEs that
\begin{align*}
Y_{1}^{\text{log}} & = y_{s} + F(y_{s}) + \frac{1}{2}F^{\prime}(y_{s})F(y_{s}) + \frac{1}{6}F^{\prime\prime}(y_{s})(F(y_{s}), F(y_{s})) + \frac{1}{6}F^{\prime}(y_{s})F^{\prime}(y_{s})F(y_{s})\\
&\hspace{8mm} + \frac{1}{24}F^{\prime}(y_{s})F^{\prime}(y_{s})F^{\prime}(y_{s})F(y_{s}) + \frac{1}{24}F^{\prime}(y_{s})F^{\prime\prime}(y_{s})(F(y_{s}), F(y_{s})) \\[2pt]
&\hspace{8mm} + \,\frac{1}{8}F^{\prime\prime}(y_{s})(F^{\prime}(y_{s})F(y_{s}),F(y_{s})) + \frac{1}{24}F^{\prime\prime\prime}(y_{s})(F(y_{s}),F(y_{s}), F(y_{s}))\\
&\hspace{8mm} + \frac{1}{24}\int_{0}^{1}(1-u)^{4}\frac{d^5}{du^{5}}\big(Y^{\text{log}}\big)(u)\,du.
\end{align*}
We shall first consider the remainder term, which can be directly estimated as follows:
\begin{align*}
\left\|\,\int_{0}^{1}(1-u)^{4}\frac{d^5}{du^{5}}\big(Y^{\text{log}}\big)(u)\,du\,\right\|_{L^{2}(\mathbb{P})} & \leq \sup_{u\in[\hspace{0.1mm}0,1]} \left\|\frac{d^5}{du^{5}}\big(Y^{\text{log}}(u)\big)\right\|_{L^{2}(\mathbb{P})}.
\end{align*}
One can define the degree of each term in the above Taylor expansion by counting
the number of times functions from $\{F, F^{\prime}, F^{\prime\prime}, \cdots\}$ appear. Therefore, after expanding the fifth derivative of $Y^{\text{log}}$ we can see
that the remainder term has a degree of five.
Since the largest component of $F$ is $f_{1}(\cdot)\hspace{0.25mm}W_{s,t}$, both $F$ and its derivatives are $O(h^{\frac{1}{2}})$.
Hence the remainder term in the above Taylor expansion will be $O(h^{\frac{5}{2}})$ as in  (\ref{highorderestimate}).
Moreover, the only terms of degree four that are not $O(h^{\frac{5}{2}})$ are those involving $W_{s,t}^{4}$.\\
It is now enough to analyse just the terms appearing in the first line of the expansion.
By substituting the formula for $F$ given by (\ref{logode}) into the first line and then rearranging
the resulting terms, we can obtain a Taylor expansion for $Y_{1}^{\text{log}}$ that resembles (\ref{stochtaylorexp}) as
\begin{align*}
F(y_{s}) & = f_{0}(y_{s})\hspace{0.25mm}h + f_{1}(y_{s})\hspace{0.25mm}W_{s,t} + \big(f_{0}^{\prime}(y_{s})f_{1}(y_{s}) - f_{1}^{\prime}(y_{s})f_{0}(y_{s})\big)\hspace{0.25mm} hH_{s,t}\\[4pt]
&\hspace{5mm} + \big(f_{0}^{\prime}(y_{s})f_{1}^{\prime}(y_{s})f_{1}(y_{s}) + f_{0}^{\prime\prime}(y_{s})(f_{1}(y_{s}), f_{1}(y_{s}))\\[4pt]
&\hspace{10mm} - 2\hspace{0.25mm}f_{1}^{\prime}(y_{s})f_{0}^{\prime}(y_{s})f_{1}(y_{s}) - 2\hspace{0.25mm}f_{1}^{\prime\prime}(y_{s})(f_{0}(y_{s}), f_{1}(y_{s}))\\
&\hspace{10mm} + f_{1}^{\prime}(y_{s})f_{1}^{\prime}(y_{s})f_{0}(y_{s}) + f_{1}^{\prime\prime}(y_{s})(f_{1}(y_{s}), f_{0}(y_{s}))\big)\hspace{-0.5mm}\left(\frac{1}{30}\hspace{0.25mm}h^{2} + \frac{3}{5}\hspace{0.25mm}hH_{s,t}^{2}\right)\hspace{-0.5mm},\\[3pt]
F^{\prime}(y_{s})F(y_{s}) & = f_{1}^{\prime}(y_{s})f_{1}(y_{s})\hspace{0.25mm} W_{s,t}^{2} + \big(f_{0}^{\prime}(y_{s})f_{1}(y_{s}) + f_{1}^{\prime}(y_{s})f_{0}(y_{s})\big)\hspace{0.25mm} hW_{s,t}\\[3pt]
&\hspace{5mm} + \big(f_{0}^{\prime}(y_{s})f_{1}^{\prime}(y_{s})f_{1}(y_{s}) - f_{1}^{\prime}(y_{s})f_{1}^{\prime}(y_{s})f_{0}(y_{s})\\[3pt]
&\hspace{10mm} + f_{0}^{\prime\prime}(y_{s})(f_{1}(y_{s}), f_{1}(y_{s})) - f_{1}^{\prime\prime}(y_{s})(f_{0}(y_{s}), f_{1}(y_{s}))\big)\hspace{0.25mm} hW_{s,t}H_{s,t}\\[3pt]
&\hspace{5mm} + f_{0}^{\prime}(y_{s})f_{0}(y_{s})\hspace{0.25mm}h^{2} + O\big(h^{\frac{5}{2}}\big),
\end{align*}\vspace{-6mm}
\begin{align*}
F^{\prime}(y_{s})F^{\prime}(y_{s})F(y_{s}) & = f_{1}^{\prime}(y_{s})f_{1}^{\prime}(y_{s})f_{1}(y_{s})\hspace{0.25mm}W_{s,t}^{3} + f_{0}^{\prime}(y_{s})f_{1}^{\prime}(y_{s})f_{1}(y_{s})\hspace{0.25mm}hW_{s,t}^{2}\\[3pt]
&\hspace{5mm} + \big(f_{1}^{\prime}(y_{s})f_{0}^{\prime}(y_{s})f_{1}(y_{s}) + f_{1}^{\prime}(y_{s})f_{1}^{\prime}(y_{s})f_{0}(y_{s})\big)\hspace{0.25mm}hW_{s,t}^{2} + O\big(h^{\frac{5}{2}}\big),\\[8pt]
F^{\prime\prime}(y_{s})\big(F(y_{s}), F(y_{s})\big) & = f_{1}^{\prime\prime}(y_{s})(f_{1}(y_{s}), f_{1}(y_{s}))\hspace{0.25mm}W_{s,t}^{3} + 2\hspace{0.25mm}f_{1}^{\prime\prime}(y_{s})(f_{0}(y_{s}), f_{1}(y_{s}))\hspace{0.25mm}hW_{s,t}^{2}\\[2pt]
&\hspace{5mm} + f_{0}^{\prime\prime}(y_{s})(f_{1}(y_{s}), f_{1}(y_{s}))\hspace{0.25mm}hW_{s,t}^{2} + O\big(h^{\frac{5}{2}}\big).
\end{align*}
Therefore, by summing the above formulae for $F$ (and its derivatives) we can derive an expansion of $Y_{1}^{\text{log}}$ in terms of $f_{0}, f_{1}$ and $\big(h, W_{s,t}, H_{s,t}\big)$ that has an $O\big(h^{\frac{5}{2}}\big)$ remainder.
By comparing this with the stochastic Taylor expansion (\ref{stochtaylorexp}), the result (\ref{logodestep}) follows.\medbreak\noindent
Arguing (\ref{parabolaodestep}) is fairly straightforward and does not require extensive computations.
Using the substitution $\wideparen{Y}_{u} = z_{\frac{1}{h}(u-s)}$ for $u\in[s,t]$, the ODE (\ref{paraode1}) can be rewritten as
\begin{align}\label{paraode2}
d\wideparen{Y}_{u} & = f_{0}\big(\hspace{0.5mm}\wideparen{Y}_{u}\big)\hspace{0.25mm}du + f_{1}\big(\hspace{0.5mm}\wideparen{Y}_{u}\big)\hspace{0.25mm}d\wideparen{W}_{u}\hspace{0.125mm},\\
\wideparen{Y}_{s} & = y_{s}\hspace{0.125mm},\nonumber
\end{align}
where $\wideparen{W}$ denotes the Brownian parabola defined by $(W_{s,t}, H_{s,t})$ on the interval $[s,t]$.\medbreak\noindent
By emulating the derivation of the Stratonovich-Taylor expansion (\ref{stochtaylorexp}), it is possible
to Taylor expand (\ref{paraode2}) in the same fashion. The only difference is that Stratonovich
integrals with respect to $W$ are replaced with Riemann-Stieltjes integrals against $\wideparen{W}$.\medbreak\noindent
In particular, by the change-of-variable formula for ODEs (exercise 3.17) given in \cite{Friz}, we see that the remainder term of such a Taylor expansion will have the below form:
\begin{align*}
\wideparen{R}\, = \sum_{\substack{i_{1},\hspace{0.25mm}\cdots,\hspace{0.25mm} i_{n}\in\{0,1\} \\ i_{1}\hspace{0.25mm}+\hspace{0.25mm}\cdots\hspace{0.25mm}+\hspace{0.25mm}i_{n}\hspace{0.25mm} =\hspace{0.35mm} 2n-4}}\int_{s\hspace{0.25mm} <\hspace{0.125mm} r_{1}\hspace{0.125mm} <\hspace{0.125mm}\cdots\hspace{0.125mm} <\hspace{0.125mm} r_{n} <\hspace{0.25mm} t}f_{i_{1},\hspace{0.25mm}\cdots,\hspace{0.25mm} i_{n}}\big(\hspace{0.5mm}\wideparen{Y}_{r_{1}}\big) - f_{i_{1},\hspace{0.25mm}\cdots,\hspace{0.25mm} i_{n}}\big(\hspace{0.5mm}\wideparen{Y}_{s}\hspace{0.25mm}\big)\,d\wideparen{W}_{r_{1}}^{i_{1}}\,\cdots\, d\wideparen{W}_{r_{n}}^{i_{n}}\hspace{0.25mm},
\end{align*}
where we have identified an additional ``\hspace{0.2mm}zero'' coordinate of $\wideparen{W}$ with time, $\wideparen{W}_{t}^{0} := t$, and for each index $(i_{1}, \cdots i_{n})$, the function $f_{i_{1},\hspace{0.25mm}\cdots,\hspace{0.25mm} i_{n}} : \mathbb{R}^{d}\rightarrow \mathbb{R}^{d}$ consists of finitely many compositions of $f_{0}, f_{1}$ along with their derivatives (and thus is Lipschitz continuous).\medbreak\noindent
Therefore each term in the expansion of (\ref{paraode2}) can be estimated in $L^{2}(\mathbb{P})$ by applying
the natural Brownian scaling to the corresponding iterated integral of $\wideparen{W}$ with time.
As before, the largest differences are the $O(h^{2})$ terms involving third order integrals.
Fortunately, iterated integrals of the Brownian parabola can be computed explicitly:
\begin{align*}
\int_{s}^{t}\int_{s}^{u}\int_{s}^{v}\circ\,dW_{r}\circ\, dW_{v}\, du
- \int_{s}^{t}\int_{s}^{u}\int_{s}^{v}\,d\wideparen{W}_{r}\, d\wideparen{W}_{v}\, du
& = L_{s,t} - \frac{3}{5}H_{s,t}^{2}\hspace{0.125mm}.
\end{align*}
The result (\ref{parabolaodestep}) is now a direct consequence of Theorem \ref{logsigrelation} along with the above.
\end{proof}
\medbreak\noindent
Theorem \ref{localerrors} shows that both methods give a one-step approximation error of $O(h^{2})$.
This means that the log-ODE and parabola-ODE methods are both locally high order;
however there is a significant difference in how these methods propagate local errors.
The reason is that the $O(h^{2})$ components of the log-ODE local errors give a martingale,
whilst the $O(h^{2})$ part for each parabola-ODE local error has non-zero expectation.
Thus the log-ODE method is globally high order whilst the parabola method is not.
However, since the parabola-ODE method is straightforward to implement and locally
high order, one could expect it to perform well compared to other low order methods.
In the numerical example, we shall see that the parabola method has the same order
of convergence as the piecewise linear approach but gives significantly smaller errors.
To conclude this section, we will present the orders of convergence for both methods.\medbreak
\begin{definition}[Strong convergence] A numerical solution $Y$ for (\ref{ogsde}) is said
to converge in a strong sense with order $\alpha$ if there exists a constant $C > 0$ such that
\begin{align*}
\left\|Y_{N} - y_{T}\right\|_{L^{2}\left(\mathbb{P}\right)} \leq Ch^{\alpha},
\end{align*}
for all sufficiently small step sizes $h = \frac{T}{N}$.
\end{definition}\medbreak
\begin{definition}[Weak convergence] A numerical solution $Y$ for (\ref{ogsde}) is said
to converge in a weak sense with order $\beta$ if for any polynomial $p$ there exists $C_{p} > 0$
such that
\begin{align*}
\big|\,\mathbb{E}\hspace{-0.25mm}\left[\,p\hspace{0.25mm}(\hspace{0.25mm}Y_{N})\right] - \mathbb{E}\hspace{-0.25mm}\left[\,p\hspace{0.25mm}(\hspace{0.25mm}y_{T})\right]\big|\leq C_{p}\,h^{\beta},
\end{align*}
for all sufficiently small step sizes $h = \frac{T}{N}$.
\end{definition}\medbreak
\begin{theorem}[Orders of convergence]
For a general SDE (\ref{ogsde}), the log-ODE method converges in a strong sense with order 1.5 and a weak sense with order 2.0.
The parabola-ODE method converges in both a strong and weak sense with order 1.0.
\end{theorem}\medbreak
\begin{proof}
Note that Theorem \ref{localerrors} establishes the Taylor expansions of both methods.
The strong convergence can then be shown as in the proof of Theorem 11.5.1 in \cite{KloePlat}.
Moreover, the proof of Theorem 11.5.1 also provides the orders of strong convergence.
Similarly weak convergence follows directly from the Taylor expansions (\ref{logodestep}) \& (\ref{parabolaodestep}),
and the rate of convergence can be shown as in the proof of Theorem 14.5.2 in \cite{KloePlat}.
\end{proof}

\section{A numerical example}
We shall demonstrate the ideas presented so far
using various discretizations of Inhomogeneous Geometric Brownian Motion (IGBM)
\begin{align}\label{IGBM}
dy_{t} = a(b-y_{t})\,dt+\sigma\hspace{0.125mm} y_{t}\,dW_{t}\hspace{0.125mm},
\end{align}
where $a \geq 0$ and $b\in \mathbb{R}$ are the mean reversion parameters and $\sigma \geq 0$ is the volatility.
As the vector fields are smooth, the SDE (\ref{IGBM}) can be expressed in Stratonovich form:
\begin{align}\label{stratIGBM}
dy_{t} = \tilde{a}(\tilde{b} - y_{t})\,dt + \sigma\hspace{0.125mm} y_{t}\circ dW_{t}\hspace{0.125mm},
\end{align}
where $\tilde{a} := a + \frac{1}{2}\sigma^{2}$ and $\tilde{b} := \frac{2ab}{2a+\sigma^{2}}$ denote the ``adjusted'' mean reversion parameters.\medbreak\noindent
IGBM is an example of a one-factor short rate model and has seen recent attention in
the mathematical finance literature as an alternative to popular models \cite{IGBMapproximation, IGBMapplication}.
IGBM is also one of the simplest SDEs that has no known method of exact simulation.
We will investigate the strong and weak convergence rates of the following methods:\medbreak
\begin{enumerate}
\item \textbf{Log-ODE method} (see definition \ref{logodedef})\smallskip\\
Since the vector fields of (\ref{stratIGBM}) give constant Lie brackets, this method becomes
\begin{align*}
Y_{k+1}^{\text{log}} & := Y_{k}^{\text{log}}e^{-\tilde{a}h + \sigma W_{t_{k},t_{k+1}}}\\
&\hspace{4mm} +  abh\left(1 - \sigma H_{t_{k},t_{k+1}}\hspace{-0.5mm}
+ \sigma^{2}\left(\frac{3}{5}H_{t_{k},t_{k+1}}^{2}\hspace{-0.5mm} + \frac{1}{30}h\right)\right)
\frac{e^{-\tilde{a}h + \sigma W_{t_{k},t_{k+1}}}  - 1}{-\tilde{a}h + \sigma W_{t_{k},t_{k+1}}}\,,\nonumber\\
Y_{0}^{\text{log}} & := y_{0}\hspace{0.125mm}.\nonumber
\end{align*}
\item \textbf{Parabola-ODE method} (see definition \ref{parabolaodedef})\smallskip\\
As the SDE (\ref{stratIGBM}) is quite analytically tractable, this method is expressible as
\begin{align*}
Y_{k+1}^{\text{para}} & := e^{-\tilde{a}h + \sigma W_{t_{k},t_{k+1}}}\left(Y_{k}^{\text{para}} 
+ ab\int_{t_{k}}^{t_{k+1}}e^{\tilde{a}\left(s-t_{k}\right) - \sigma \wideparen{W}_{t_{k}, s}}\,ds\right),\\
Y_{0}^{\text{para}} & := y_{0}\hspace{0.125mm}.\nonumber
\end{align*}
The integral above will be computed by $3$-point Gauss-Legendre quadrature.\medskip
\item \textbf{Piecewise linear method} (see \cite{WongZakai} for definition and proof of convergence)\smallskip\\
Just as above, this method can be simplified to give a straightforward formula.
\begin{align*}
Y_{k+1}^{\text{lin}} & := Y_{k}^{\text{lin}}e^{-\tilde{a}h + \sigma W_{t_{k},t_{k+1}}}
+  abh\,\frac{e^{-\tilde{a}h + \sigma W_{t_{k},t_{k+1}}}  - 1}{-\tilde{a}h + \sigma W_{t_{k},t_{k+1}}}\,,\\
Y_{0}^{\text{lin}} & := y_{0}\hspace{0.125mm}.\nonumber
\end{align*}
\item \textbf{Milstein method} (see section 6 of \cite{Higman} and section 10.3 of \cite{KloePlat} for overviews)\smallskip\\
For this method, we shall take the positive part to guarantee non-negativity.
\begin{align*}
Y_{k+1}^{\text{mil}} & := \Big(Y_{k}^{\text{mil}} + \tilde{a}(\tilde{b} - Y_{k}^{\text{mil}})h
+ \sigma \,Y_{k}^{\text{mil}}\,W_{t_{k}, t_{k+1}}\hspace{-0.5mm}
+ \frac{1}{2}\sigma^{2}\,Y_{k}^{\text{mil}}\,W_{t_{k}, t_{k+1}}^{2}\Big)^{+},\\
Y_{0}^{\text{mil}} & := y_{0}\hspace{0.125mm}.\nonumber
\end{align*}
\item \textbf{Euler-Maruyama method} (see sections 4, 5 of \cite{Higman} and section 10.2 of \cite{KloePlat})\smallskip\\
Just as above, we take the positive part of each step to ensure non-negativity.
\begin{align*}
Y_{k+1}^{\text{eul}} & := \Big(Y_{k}^{\text{eul}} + a(b - Y_{k}^{\text{eul}})h
+ \sigma \,Y_{k}^{\text{eul}}\,W_{t_{k}, t_{k+1}}\Big)^{+},\\
Y_{0}^{\text{eul}} & := y_{0}\hspace{0.125mm}.\nonumber
\end{align*}
\end{enumerate}
\medbreak\noindent
Note that the explicit formula for the log-ODE method comes from the Lie brackets:
\begin{align*}
[f_{1}, f_{0}](y) & = f_{0}^{\prime}(y)f_{1}(y) - f_{1}^{\prime}(y)f_{0}(y)\\[3pt]
& = -\tilde{a}\sigma y - \tilde{a}\sigma(\tilde{b} - y)\\[3pt]
& = -ab\sigma\hspace{0.25mm},\\[6pt]
[f_{1}, [f_{1}, f_{0}]](y) & = f_{0}^{\prime}(y)f_{1}^{\prime}(y)f_{1}(y) - 2f_{1}^{\prime}(y)f_{0}^{\prime}(y)f_{1}(y) + f_{1}^{\prime}(y)f_{1}^{\prime}(y)f_{0}(y)\\[2pt]
&\hspace{8mm} + f_{0}^{\prime\prime}(y)(f_{1}(y), f_{1}(y)) - 2\hspace{0.25mm}f_{1}^{\prime\prime}(y)(f_{0}(y), f_{1}(y)) + f_{1}^{\prime\prime}(y)(f_{1}(y), f_{0}(y))\\[3pt]
& = -\tilde{a}\sigma^{2}y + 2\hspace{0.25mm}\tilde{a}\sigma^{2}y + \tilde{a}\sigma^{2}(\tilde{b} - y)\\[3pt]
& = ab\sigma^{2}\hspace{0.25mm},
\end{align*}
and the formula for the parabola-ODE method was derived using a change of variable.\medbreak\noindent
The Euler-Maruyama and Milstein methods are included in the numerical experiment 
as benchmarks to test how the proposed methods compare to well-known methods.
As before, we will be discretizing the SDE over a uniform partition with mesh size $h$.\vspace{-3.5mm}
\begin{figure}[h]\label{IGBMsamplepaths}
\centering
\includegraphics[width=0.95\textwidth]{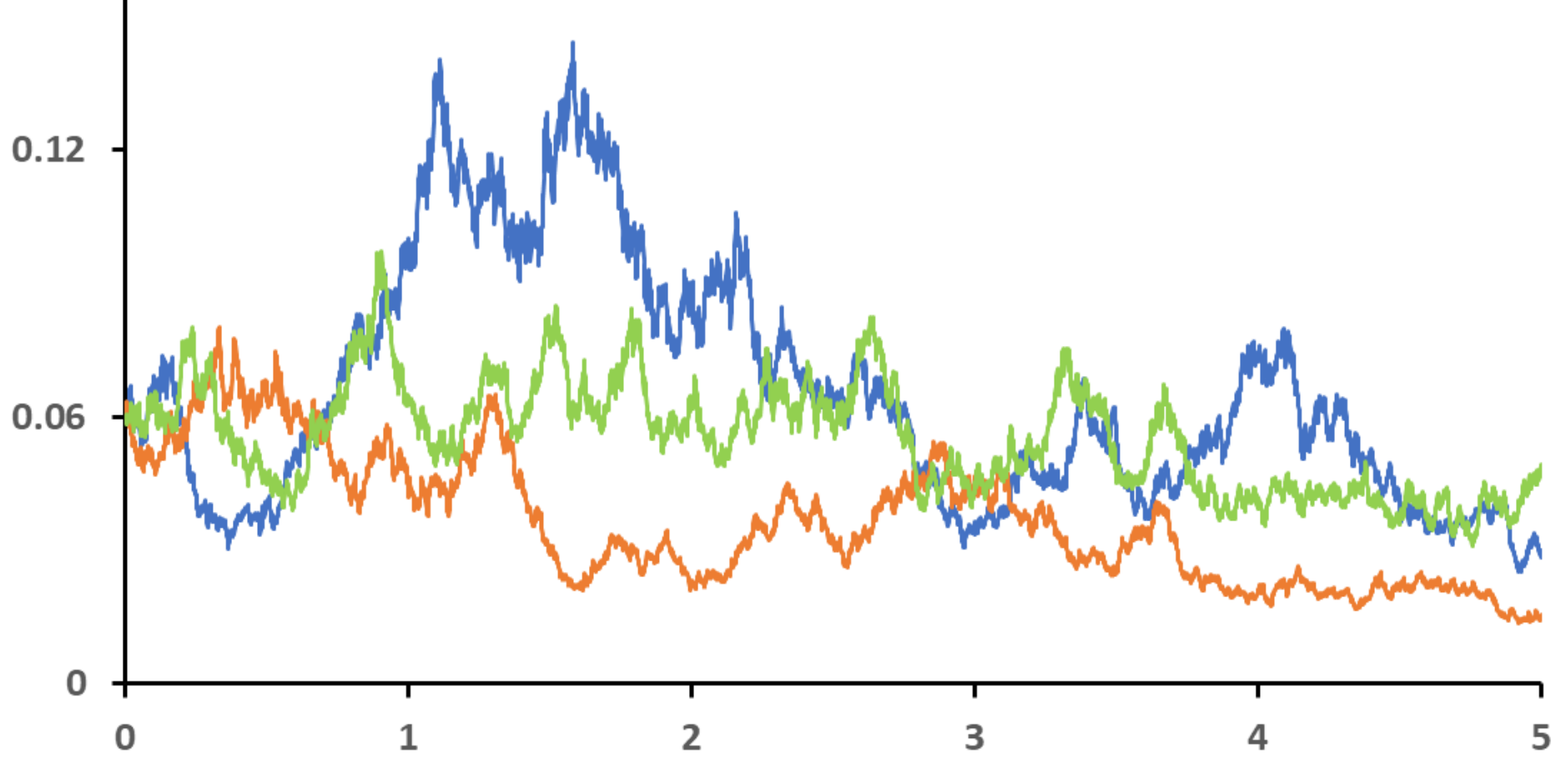}\vspace{-2.5mm}
\caption{Log-ODE sample paths of IGBM where $a=0.1$, $b=0.04$,  $\sigma = 0.6$ and $y_{0} = 0.06$.\newline\hspace*{5mm} The above sample paths experience larger fluctuations the further away from zero they are.}
\end{figure}\medbreak\vspace{-2.5mm}\noindent
Below is the definition of the error estimators used to analyse the numerical methods.\medbreak
\begin{definition}[Strong and weak error estimators]\label{errorestimators} For each $N\geq 1$, let $Y_{N}$
denote a numerical solution of (\ref{IGBM}) computed at time $T$ using a fixed step size $h = \frac{T}{N}$.
We can define the following estimators for quantifying strong and weak convergence: 
\begin{align}
S_{N} & := \sqrt{\hspace{0.25mm}\mathbb{E}\hspace{-0.25mm}\left[\left(Y_{N} - Y_{T}^{fine}\right)^{2}\hspace{0.5mm}\right]}\hspace{0.25mm},\label{strongestimator}\\
E_{N} & := \left|\,\mathbb{E}\big[\big(Y_{N} - b\hspace{0.25mm}\big)^{+}\hspace{0.25mm}\big] - \mathbb{E}\big[\big(Y_{T}^{fine} - b\hspace{0.25mm}\big)^{+}\hspace{0.25mm}\big]\right|,\label{weakestimator}
\end{align}
where the above expectations are approximated by standard Monte-Carlo simulation
and $Y_{T}^{fine}$ is the numerical solution of (\ref{IGBM}) obtained at time $T$ using the log-ODE method
with a ``\hspace{0.5mm}fine'' step size of $\min\left(\frac{h}{10}, \frac{T}{1000}\right)$. The fine step size is chosen so that the
$L^{2}(\mathbb{P})$ error between $Y_{T}^{fine}$ and the true solution $y$ is negligible compared to $S_{N}$.
Note that $Y_{N}$ and $Y_{T}^{fine}$ are both computed with respect to the same Brownian paths.
\end{definition}\medbreak\noindent
In this numerical example, we shall use the same parameter values as in \cite{IGBMapproximation}, namely
$a = 0.1$, $b=0.04$, $\sigma = 0.6$ and $y_{0} = 0.06$. We will also fix the time horizon at $T = 5$.\medbreak\noindent
We will now present our results for the numerical experiment that is described above.
(Code for this example can be found at \href{https://github.com/james-m-foster/igbm-simulation}{github.com/james-m-foster/igbm-simulation})
\begin{figure}[h]
\centering
\hspace*{-1mm}\includegraphics[width=1.025\textwidth]{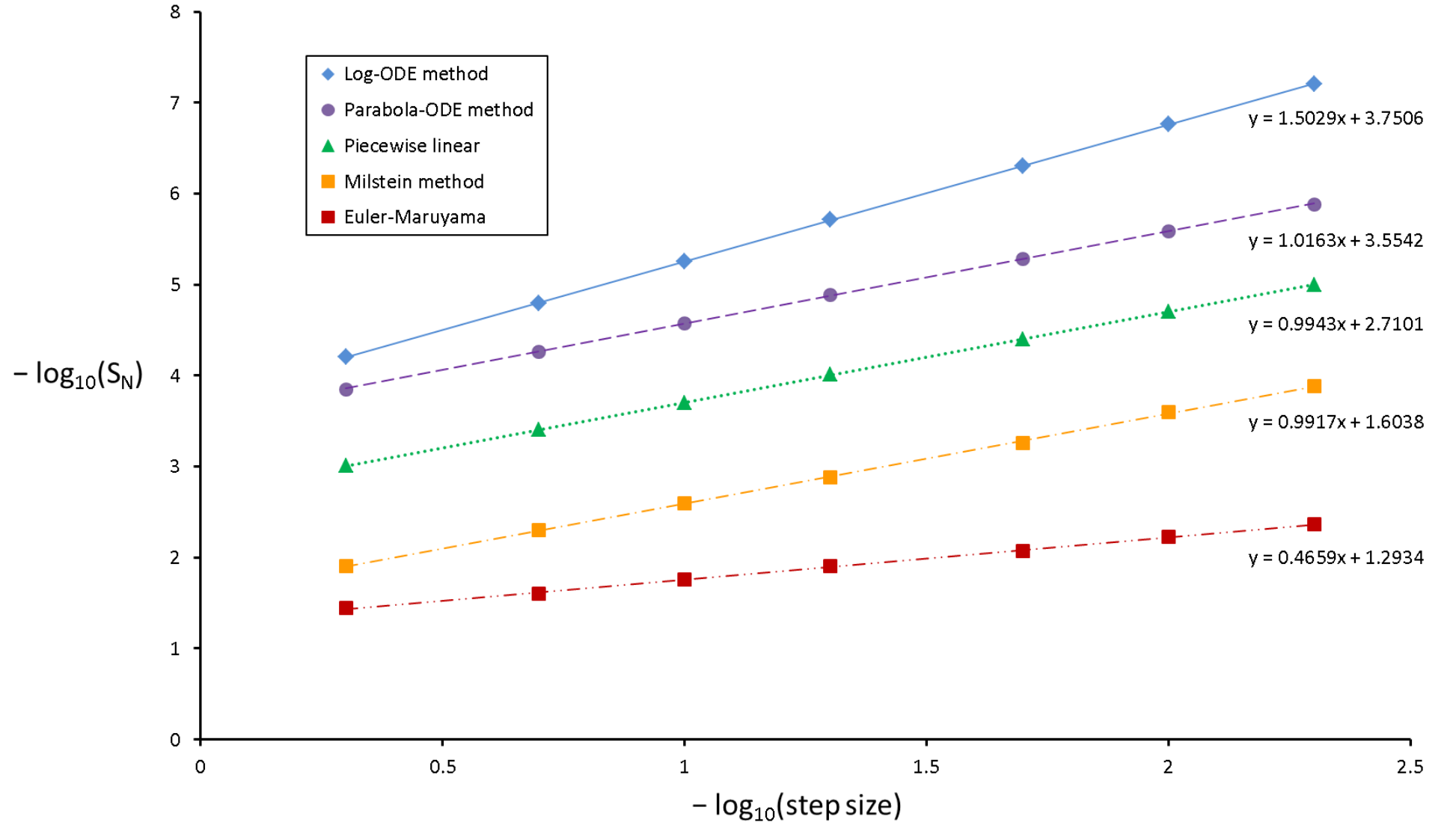}
\caption{$S_{N}$ computed with 100,000 sample paths as a function of step size $h = \frac{T}{N}$.}\label{strongconvergediagram}
\end{figure}\smallbreak\noindent
From the above graph we see that the log-ODE method is by far the most accurate.
This is epitomized by the fact that the numerical error produced by $100$ steps of the
log-ODE method is comparable to the error of the parabola method with $1000$ steps.
In addition, whilst there are three methods that share the same order of convergence
it is evident there are magnitudes of difference between their respective accuracies. 
For example, the parabola method is seven times more accurate than piecewise linear.
As one might expect, the Euler-Maruyama and Milstein schemes both perform poorly.\medbreak\noindent
Nevertheless, in order to truly measure the performance of these numerical methods,
we should consider the computational costs required for achieving a specified accuracy.\vspace{-1mm}
\begin{table}[tbhp]
\caption{\hspace{7mm}Estimated simulation times for computing 100,000 sample paths that achieve a\newline
\centerline{given accuracy using a single-threaded C++ program on a desktop computer.}}
\begin{center}
\begin{tabular}{cccccc}
  \cline{2-6}
  \multicolumn{1}{c}{} & \multicolumn{1}{|c|}{} & \multicolumn{1}{c|}{} & \multicolumn{1}{c|}{} & \multicolumn{1}{c|}{} & \multicolumn{1}{c|}{}  \\ [-1em]
  \multicolumn{1}{c}{} & \multicolumn{1}{|c|}{Log-ODE} & \multicolumn{1}{c|}{Parabola} &
  \multicolumn{1}{c|}{\hspace{0.5mm}Linear\hspace*{0.5mm}} & \multicolumn{1}{c|}{Milstein} & \multicolumn{1}{c|}{Euler} \\ \hline
  \multicolumn{1}{|c}{} & \multicolumn{1}{|c|}{} & \multicolumn{1}{c|}{} & \multicolumn{1}{c|}{}  & 
  \multicolumn{1}{c|}{} & \multicolumn{1}{c|}{} \\ [-1em]
  \multicolumn{1}{|c}{Estimated time to achieve} & \multicolumn{1}{|c|}{0.179}  &  
  \multicolumn{1}{c|}{0.405} & \multicolumn{1}{c|}{1.47} & \multicolumn{1}{c|}{15.4} & \multicolumn{1}{c|}{0.437} \\ [0.1em]
  \multicolumn{1}{|c}{an accuracy of $S_{N} = 10^{-4}$} & \multicolumn{1}{|c|}{(s)} & \multicolumn{1}{c|}{(s)} & 
  \multicolumn{1}{c|}{(s)}  & \multicolumn{1}{c|}{(s)} &   \multicolumn{1}{c|}{(days)} \\
  \multicolumn{1}{|c}{} & \multicolumn{1}{|c|}{}  &  
  \multicolumn{1}{c|}{} & \multicolumn{1}{c|}{} & \multicolumn{1}{c|}{} & \multicolumn{1}{c|}{} \\ [-1.1em] \hline
   \multicolumn{1}{|c}{} & \multicolumn{1}{|c|}{} & \multicolumn{1}{c|}{} & \multicolumn{1}{c|}{}  & 
  \multicolumn{1}{c|}{} & \multicolumn{1}{c|}{} \\ [-1em]
  \multicolumn{1}{|c}{Estimated time to achieve} & \multicolumn{1}{|c|}{0.827}  &  
  \multicolumn{1}{c|}{3.90} & \multicolumn{1}{c|}{14.9} & \multicolumn{1}{c|}{157} & \multicolumn{1}{c|}{61.2} \\ [0.1em]
  \multicolumn{1}{|c}{an accuracy of $S_{N} = 10^{-5}$} & \multicolumn{1}{|c|}{(s)} & \multicolumn{1}{c|}{(s)} & 
  \multicolumn{1}{c|}{(s)}  & \multicolumn{1}{c|}{(s)} &   \multicolumn{1}{c|}{(days)}\\ [0.1em] \hline
\end{tabular}
\end{center}\medbreak\medbreak
The above times are estimated by the graph shown in Fig \ref{strongconvergediagram} and the following table:
\end{table}
\begin{table}[tbhp]
\caption{\hspace{8mm}Simulation times for computing 100,000 sample paths with $100$ steps per path\newline
\centerline{using a single-threaded C++ program on a desktop computer.}}\label{simtimes}
\begin{center}
\begin{tabular}{cccccc}
  \cline{2-6}
  \multicolumn{1}{c}{} & \multicolumn{1}{|c|}{} & \multicolumn{1}{c|}{} & \multicolumn{1}{c|}{} & \multicolumn{1}{c|}{} & \multicolumn{1}{c|}{} \\ [-1em]
  \multicolumn{1}{c}{} & \multicolumn{1}{|c|}{Log-ODE} & \multicolumn{1}{c|}{Parabola} &
  \multicolumn{1}{c|}{\hspace{0.5mm}Linear\hspace*{0.5mm}} &  \multicolumn{1}{c|}{Milstein} & \multicolumn{1}{c|}{\hspace{0.5mm}Euler\hspace*{0.5mm}} \\ \hline
  \multicolumn{1}{|c}{} & \multicolumn{1}{|c|}{} & \multicolumn{1}{c|}{}  & 
  \multicolumn{1}{c|}{} &  \multicolumn{1}{c|}{} & \multicolumn{1}{c|}{} \\ [-1em]
  \multicolumn{1}{|c}{Computation time (s) } & \multicolumn{1}{|c|}{2.44}  &  
  \multicolumn{1}{c|}{2.95} & \multicolumn{1}{c|}{1.48} & 
  \multicolumn{1}{c|}{1.18} & \multicolumn{1}{c|}{1.17} \\ [-1.1em]
  \multicolumn{1}{|c}{} & \multicolumn{1}{|c|}{} & \multicolumn{1}{c|}{} & 
  \multicolumn{1}{c|}{} & \multicolumn{1}{c|}{}  & \multicolumn{1}{c|}{} \\  \hline
\end{tabular}
\end{center}\medbreak\medbreak
Finally, we will investigate the rates of weak convergence for these numerical methods.
\end{table}
\begin{figure}[h]\label{weakconvergediagram}
\centering
\hspace*{-0.5mm}\includegraphics[width=1\textwidth]{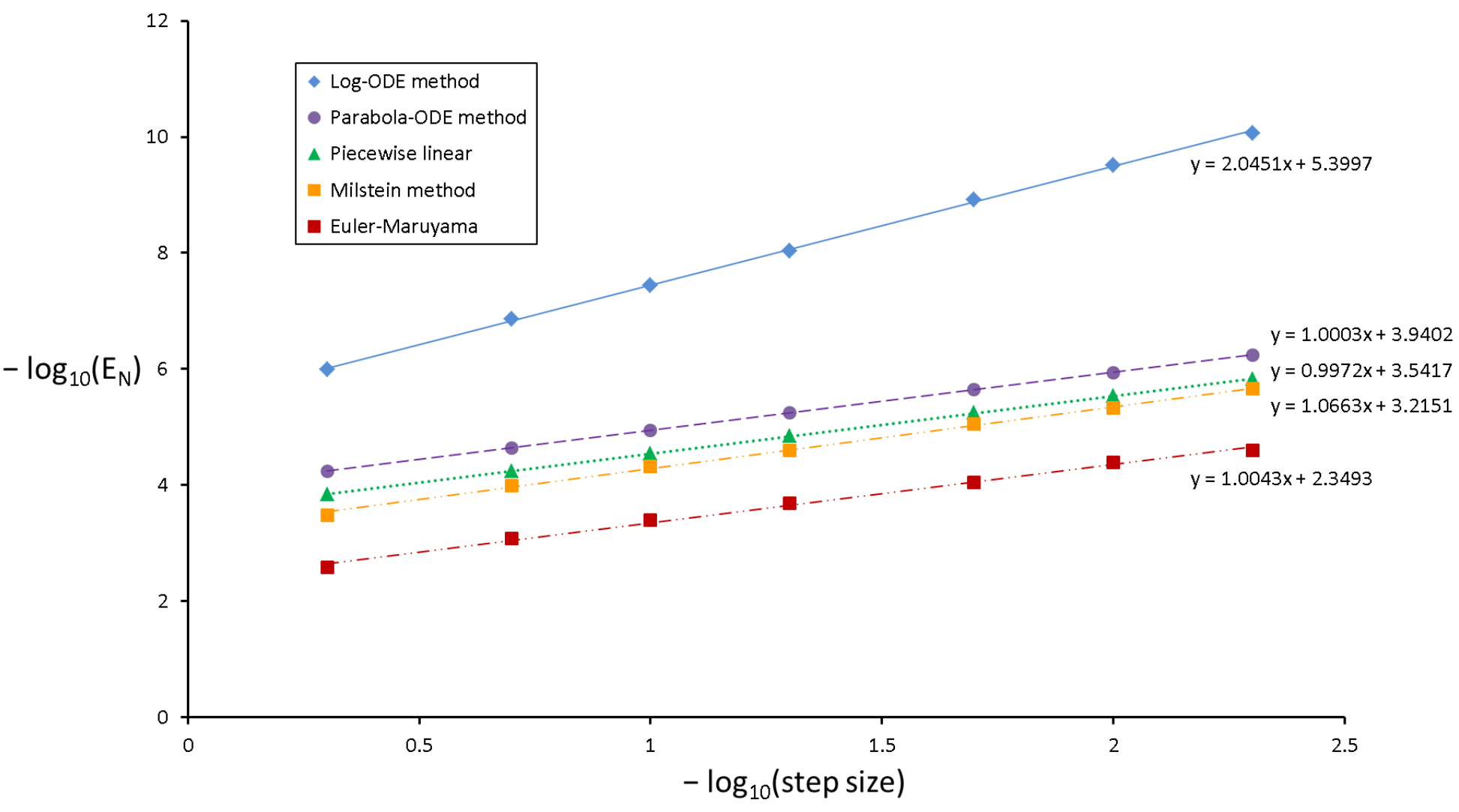}\vspace{-1.5mm}
\caption{$E_{N}$ computed with 500,000 sample paths as a function of step size $h = \frac{T}{N}$.}
\end{figure}\newpage\noindent
The above graph demonstrates that the log-ODE method is especially well-suited for 
weak approximation as it achieves a second order convergence rate in this example.
Surprisingly, the middle three methods exhibit almost identical convergence profiles.
As before, we can estimate the computational time needed to achieve given accuracies.\vspace{-2mm}
\begin{table}[tbhp]
\caption{\hspace{7mm}Estimated simulation times for computing 100,000 sample paths that achieve a\newline
\centerline{given accuracy using a single-threaded C++ program on a desktop computer.}}
\begin{center}
\begin{tabular}{cccccc}
  \cline{2-6}
  \multicolumn{1}{c}{} & \multicolumn{1}{|c|}{} & \multicolumn{1}{c|}{} & \multicolumn{1}{c|}{} & \multicolumn{1}{c|}{} & \multicolumn{1}{c|}{}  \\ [-1em]
  \multicolumn{1}{c}{} & \multicolumn{1}{|c|}{Log-ODE} & \multicolumn{1}{c|}{Parabola} &
  \multicolumn{1}{c|}{\hspace{0.5mm}Linear\hspace*{0.5mm}} & \multicolumn{1}{c|}{Milstein} & \multicolumn{1}{c|}{Euler} \\ \hline
  \multicolumn{1}{|c}{} & \multicolumn{1}{|c|}{} & \multicolumn{1}{c|}{} & \multicolumn{1}{c|}{}  & 
  \multicolumn{1}{c|}{} & \multicolumn{1}{c|}{} \\ [-1em]
  \multicolumn{1}{|c}{Estimated time to achieve} & \multicolumn{1}{|c|}{$<$ 0.240}  &  
  \multicolumn{1}{c|}{1.69} & \multicolumn{1}{c|}{2.15} & \multicolumn{1}{c|}{2.78} & \multicolumn{1}{c|}{25.5} \\ [0.1em]
  \multicolumn{1}{|c}{an accuracy of $E_{N} = 10^{-5}$} & \multicolumn{1}{|c|}{(s)} & \multicolumn{1}{c|}{(s)} & 
  \multicolumn{1}{c|}{(s)}  & \multicolumn{1}{c|}{(s)} &   \multicolumn{1}{c|}{(s)} \\
  \multicolumn{1}{|c}{} & \multicolumn{1}{|c|}{}  &  
  \multicolumn{1}{c|}{} & \multicolumn{1}{c|}{} & \multicolumn{1}{c|}{} & \multicolumn{1}{c|}{} \\ [-1.1em] \hline
   \multicolumn{1}{|c}{} & \multicolumn{1}{|c|}{} & \multicolumn{1}{c|}{} & \multicolumn{1}{c|}{}  & 
  \multicolumn{1}{c|}{} & \multicolumn{1}{c|}{} \\ [-1em]
  \multicolumn{1}{|c}{Estimated time to achieve} & \multicolumn{1}{|c|}{0.240}  &  
  \multicolumn{1}{c|}{16.9} & \multicolumn{1}{c|}{21.6} & \multicolumn{1}{c|}{24.1} & \multicolumn{1}{c|}{252} \\ [0.1em]
  \multicolumn{1}{|c}{an accuracy of $E_{N} = 10^{-6}$} & \multicolumn{1}{|c|}{(s)} & \multicolumn{1}{c|}{(s)} & 
  \multicolumn{1}{c|}{(s)}  & \multicolumn{1}{c|}{(s)} &   \multicolumn{1}{c|}{(s)}\\ [0.1em] \hline
\end{tabular}
\end{center}
\end{table}\smallbreak\vspace{-2.5mm}\noindent
We expect the log-ODE and parabola methods to have about twice the computational
cost as the other methods because each step requires generating two random variables.
Table \ref{simtimes} confirms this and thus sampling may be a bottleneck for these methods.
So overall, the numerical evidence supports our claim that the high order log-ODE
method is currently a state-of-the-art method for the pathwise discretization of IGBM.

\section{Conclusion} There are primarily three new results established in this paper:\medbreak
\begin{itemize}
\item \textbf{An efficient strong polynomial approximation of Brownian motion}\smallbreak
The main result allows one to construct a ``smoother'' Brownian motion as a
finite sum of $(\text{-}1,\text{-}1)$-Jacobi polynomials with independent Gaussian weights.
Moreover, it was shown that the approximation is optimal in a weighted $L^{2}(\mathbb{P})$
sense and the surrounding noise is an independent centered Gaussian process.\medbreak
\item \textbf{Unbiased approximation of third order Brownian iterated integrals}\smallbreak
Iterated integrals of Brownian motion and time are important objects in the
study of SDEs as they appear naturally within stochastic Taylor expansions.
We have derived the $L^{2}(\mathbb{P})$-optimal estimator for a class of such integrals that is
measurable with respect to the path's increment and space-time L\'{e}vy area.\medbreak
\item \textbf{Simulation of Inhomogeneous Geometric Brownian Motion (IGBM)}\smallbreak
IGBM is a mean-reverting short rate model used in mathematical finance and
also one of the simplest SDEs that has no known method of exact simulation.\\
By incorporating the new iterated integral estimator into the log-ODE method
we have developed a high order state-of-the-art numerical method for IGBM.\bigbreak
\end{itemize}
Furthermore, the results of this paper naturally lead to the following open questions:\medbreak
\begin{itemize}
\item Which weight functions give ``\hspace{0.125mm}explicit eigenfunctions'' for Brownian motion?
      (For example, we could try $w(x) = x$ or $w(x) = \frac{1}{x}$ with $K_{W}(s,t) = \min(s,t)$)\medbreak
\item Is it possible to generalize the main theorem to fractional Brownian motion?\medbreak
\item What are the most efficient Runge-Kutta methods for general one-dimensional
      SDEs that correctly use the  new estimator for third order iterated integrals?\medbreak
\item Is this polynomial expansion optimal for approximating L\'{e}vy area? (see \cite{Dickinson})\medbreak
\item Which conditional moments can be computed for a given stochastic integral?\medbreak
\item How might we construct a piecewise linear path $\gamma$ with the below properties?
\begin{align*}
&1.\,\,\,\, \gamma_{s} = W_{s}\hspace{0.25mm},\hspace{1.75mm} \gamma_{t} = W_{t}\hspace{0.25mm}.\\[5pt]
&2.\,\,\, \int_{s}^{t}\gamma_{s,u}\,du = \int_{s}^{t}W_{s,u}\,du\hspace{0.25mm}.\\[2pt]
&3.\,\,\, \int_{s}^{t}\gamma_{s,u}^{2}\,du = \mathbb{E}\hspace{-0.25mm}\left[\hspace{0.25mm}\int_{s}^{t}W_{s,u}^{2}\,du\,\Big|\,\cdots\right]\hspace{-0.5mm}.
\end{align*}
\item Would this method of construction lead to effective cubature paths? (see \cite{Cubature})\medbreak
Given such a path, we can approximate (\ref{ogsde}) with a ``\hspace{0.125mm}piecewise linear'' ODE.
\begin{align}\label{piecelinode}
\frac{dY}{du} = f_{0}(\hspace{0.125mm}Y) + f_{1}(\hspace{0.125mm}Y)\hspace{0.5mm}\frac{d\gamma}{du}\hspace{0.5mm}.
\end{align}
(Along each piece of $\gamma$, we would discretize (\ref{piecelinode}) using an appropriate solver)\medbreak
\item How effective is the above piecewise linear ODE method for simulating SDEs?\medbreak
\item Can we extend the approximations given in this paper to the SPDE setting?
\end{itemize}

\newpage\bigbreak

\bibliographystyle{amsplain}

\end{document}